\documentclass[preprint,1p,times]{elsarticle}

\usepackage{amsmath,amssymb,amsthm,amscd,amsxtra,amsfonts}
\usepackage{epsf,graphicx,epsfig,color,latexsym,multirow}
\usepackage{algorithm}
\usepackage{algorithmic}
\usepackage{subfigure}
\theoremstyle{definition}
\newtheorem{lemma}{Lemma}[section]
\newtheorem{remark}{Remark}[section]
\newtheorem{example}{Example}[section]
\newtheorem{theorem}{Theorem}[section]

\providecommand*{\Dist}[2]{\operatorname{dist}({#1};{#2})}   
\providecommand*{\Dist}[2]{\Dist{#1}{#2}}

\newcommand{\Bk}{{\boldsymbol{k}}}

\newcommand{\Bx}{{\boldsymbol{x}}}

\newcommand{\BB}{{\boldsymbol{B}}}

\newcommand{\BG}{{\boldsymbol{G}}}

\newcommand{\BL}{{\boldsymbol{L}}}

\newcommand{\Cc}{\mathcal{C}}

\newcommand{\Ce}{\mathcal{E}}

\newcommand{\Ci}{\mathcal{I}}

\newcommand{\Ck}{\mathcal{K}}

\newcommand{\Cp}{\mathcal{P}}

\newcommand{\bbN}{\mathbb{N}}

\newcommand{\bbR}{\mathbb{R}}

\newcommand{\bbZ}{\mathbb{Z}}

\newcommand*{\SP}[2]{\left<{#1},{#2}\right>} 

\providecommand*{\wh}[1]{\widehat{#1}}

\newcommand*{\N}[1]{\left\|{#1}\right\|}     
\newcommand*{\IP}[2]{\left({#1},{#2}\right)} 
\newcommand*{\SN}[1]{\left|{#1}\right|}      
\renewcommand*{\SP}[2]{\left({#1},{#2}\right)} 
\newcommand*{\SET}[1]{\left\{{#1}\right\}}

\newcommand{\DX}{\,\mathrm{d}\Bx}

\newcommand{\dmea}[1]{\,\mathrm{d}{#1}}

\newcommand*{\Lp}[2][\defaultdomain]{L^{#2}({#1})}
\newcommand*{\Lpv}[2][\defaultdomain]{\BL^{#2}({#1})}
\newcommand*{\NLp}[3][\defaultdomain]{\N{#2}_{\Lp[#1]{#3}}}

\newcommand*{\Ltwo}[1][\defaultdomain]{\Lp[#1]{2}}
\newcommand*{\Ltwov}[1][\defaultdomain]{\Lpv[#1]{2}}

\newcommand*{\NLtwo}[2][\defaultdomain]{\NLp[#1]{#2}{2}}

\newcommand*{\Hm}[2][\defaultdomain]{H^{#2}({#1})}

\newcommand*{\NHm}[3][\defaultdomain]{{\N{#2}}_{\Hm[{#1}]{#3}}}

\newcommand*{\Hmper}[2][\defaultdomain]{H_{p}^{#2}({#1})}

\newcommand*{\Unicont}[1][\overline\defaultdomain]{C({#1})}
\newcommand*{\Nunicont}[2][\overline\defaultdomain]{\N{#2}_{\Unicont[#1]}}

\newcommand{\D}{\mathrm{d}}

\newcommand{\be}{\begin{eqnarray}}
\newcommand{\ee}{\end{eqnarray}}

\newcommand{\ben}{\begin{eqnarray*}}
\newcommand{\een}{\end{eqnarray*}}

\providecommand*{\diag}[1]{\operatorname{diag}\left({#1}\right)}

\newcommand*{\dt}[1]{\frac{\D {#1}}{\D t}}

\newcommand*{\dphi}[1]{{\delta_\phi {#1}}}

\newcommand*{\ddv}[2]{\frac{\D^2 {#1}}{\D {#2}^2}}

\graphicspath{{figures/}}

\begin{document}
\title{ Spectral Deferred Correction Method for Landau--Brazovskii Model with Convex Splitting Technique}
\author[add1]{Donghang Zhang}\ead{zdh@lsec.cc.ac.cn}
\author[add1,add2,add3]{Lei Zhang\corref{cor1}}\ead{zhangl@math.pku.edu.cn}
\address[add1]{Beijing International Center for Mathematical Research, Peking University, Beijing 100871, China.}
\address[add2]{Center for Quantitative Biology, Peking University, Beijing 100871, China.}
\address[add3]{Center for Machine Learning Research, Peking University, Beijing 100871, China.}
\cortext[cor1]{Corresponding author.}

\begin{abstract}
The Landau--Brazovskii model is a well-known Landau model for finding the complex phase structures in microphase-separating systems ranging from block copolymers to liquid crystals. It is critical to design efficient numerical schemes for the Landau--Brazovskii model with energy dissipation and mass conservation properties. Here, we propose a mass conservative and energy stable scheme by combining the spectral deferred correction (SDC) method with the convex splitting technique to solve the Landau--Brazovskii model efficiently. 
An adaptive correction strategy for the SDC method is implemented to reduce the cost time and {preserve} energy stability.
Numerical experiments, including two- and three-dimensional periodic crystals in Landau--Brazovskii model, are presented to show the efficiency of the proposed numerical method.
\end{abstract}

\maketitle

\section{Introduction} \label{sec1}
The Landau--Brazovskii (LB) model is a generic model to describe phase transitions driven by a short-wavelength instability between the disordered phases and the ordered phases \cite{bra75}. This model has been widely used to simulate several physical systems, such as block copolymers \cite{shi96,zh08}, liquid crystals \cite{kat93, wan21} and other microphase-separating systems \cite{lif97,yao22}.
More concretely, the energy functional of the LB Model is given by
\begin{align}\label{LB-energy}
\Ce(\phi) = \int_\Omega\left\{ \frac{1}2\left((\Delta \phi +  \phi\right)^2  -  \frac\alpha{2!}\phi^2 + \frac{1}{4!}\phi^4 - \frac\gamma{3!}\phi^3 \right\}\DX,\quad
\end{align}
where the order parameter field $\phi(\Bx)$ is the real-valued periodic function that measures the order of the system in $\bbR^d$, 
$\SN{\Omega}$ is the volume of domain $\Omega\subset\bbR^d$, 
and  $\alpha,\gamma$ are adjustable parameters . 
Compared with the typical Swift-Hohenberg model \cite{swi77} with double-well bulk energy,
the LB energy functional includes a cubic term that can be used to characterize the asymmetry of the order phase.
Note that $\Ce$ is invariant under the transformation $\phi\to-\phi,\gamma\to-\gamma$. 
For convenience, we suppose $\gamma \ge 0$.
Moreover, to conserve the number of particles in the system, $ \phi $  satisfies mass conservation as follow:
\begin{align}\label{mass}
\bar\phi:=\frac1{|\Omega|}\int_\Omega\phi\DX = 0.
\end{align}

To find the stationary states of the LB model, the Allen-Cahn gradient flow of the LB model (AC-LB) reads as
\begin{align}\label{AC-LB}
\partial_t\phi = -\dphi{\Ce}(\phi) + \beta(\phi),\quad 
\beta(\phi):=\frac1{\SN{\Omega}} \int_\Omega\left((1 - \alpha)\phi + \frac{\phi^3}{3!} - \frac{\gamma}{2} \phi^2\right)\DX,
\end{align}
where $\dphi{\Ce}$ is the first order variational derivative of $\Ce$ with respect to $\phi$, 
$\partial_t $ is the partial derivative with respect to $t$, 
and the last term $\beta(\phi)$ is the Lagrange multiplier to conserve the total mass of $\phi$ \cite{rub92,zh19}. 

It is straightforward to show that the equation \eqref{AC-LB} satisfies the mass conservation and energy dissipative law. First,
by taking the inner product of \eqref{AC-LB} with $1$ and using integration by parts, we have
\begin{align}\label{mass-con}
\dt{}\bar\phi = 0.
\end{align}
Next, by taking the inner product of \eqref{AC-LB} with $\partial_t\phi$ and using integration by parts and the mass conservation \eqref{mass-con}, we obtain the following energy dissipative law:
\begin{align}\label{dissipation}
\dt{}\Ce(\phi) = -\int_\Omega (\partial_t\phi)^2 \DX <0.
\end{align}
Therefore, the goal of this paper is to develop the efficient numerical method for the AC-LB equation \eqref{AC-LB} while keeping the mass conservation \eqref{mass-con} and desired energy dissipation \eqref{dissipation} during the iterative process.
Then the energy minimizer of the LB model \eqref{LB-energy} is obtained with a proper choice of initialization.

There are many efforts devoted to designing the numerical schemes for the nonlinear gradient flow equations with energy dissipation and mass conservation properties.
For example, typical energy stable schemes to gradient flow include the convex splitting methods \cite{eyr98}, the exponential time differencing schemes \cite{du21}, the stabilized factor methods \cite{she10}, and invariant energy quadrature \cite{yan16} and scalar auxiliary variable methods \cite{she19} for modified energy.
What's more, the Cahn-Hilliard gradient flow can automatically ensure the mass conservation property. 
There have been many works based on Cahn-Hilliard gradient flow, such as the phase field crystal model \cite{hu09,wis09}, modified phase field crystal model \cite{bas13-1,bas13-2,wan11}, and square phase field crystal model \cite{che19,Wan21}.
Numerically, the gradient flow needs to be discretized in both the space and time domains. The typical spatial discretization techniques
include 
the finite difference method \cite{hu09,wis09,wan11,xu21}, the finite element method \cite{du08,wan20} and the Fourier pseudo-spectral method \cite{che19,jk14,she98,wan21,yin21}
.

To calculate the stationary states of the LB model, an efficient numerical method was developed by using the Fourier expansion of order parameter to find the meta-stable and stable phases in the diblock copolymer system \cite{zh08}. 
A second-order invariant energy quadrature approach with the stabilization technique was proposed to keep the required accuracy while using large time steps \cite{zh19}.  
To make the numerical scheme linear while preserving the nonlinear energy stability, a second order the scalar auxiliary variable method for the square phase field crystal equation was proposed and analyzed \cite{Wan21}. 
The high order linear convex splitting schmes were used in the expitaxial thin film model without slope selection. 
The energy stability and convergence analysis can be found in \cite{che12,ha21,li18,men20}
By using the optimization techniques, Jiang et al. proposed adaptive accelerated Bregman proximal gradient methods for phase field crystal equation\cite{jk20}.

In this paper, we propose an efficient mass conservative and energy stable scheme for the LB model by combining the convex splitting technique with the spectral deferred correction (SDC) method. 
The mass conservative and energy stable properties are proved for the linear convex splitting method for the AC-LB equation \eqref{AC-LB}.
The SDC method was first introduced to solve initial value ordinary differential equations (ODEs) in \cite{dut00}, and the central idea of the SDC method is to convert the original ODEs into the corresponding Picard equation and then use a deferred correction procedure in the integral formulation to achieve higher-order accuracy iteratively.
We choose the SDC method combined with the convex splitting technique for the following reasons: 
Iteration loops can improve formal accuracy flexibly and straightforwardly; 
the SDC method was designed to handle stiff systems, such as singularly nonlinear equations. 
Moreover, an adaptive correction strategy for the SDC method is proposed to increase the rate of convergence and reduce the cost time.
Both two- and three-dimensional periodic crystals in the LB model are shown in numerical examples to demonstrate the accuracy and efficiency of the proposed approach.

The rest of this paper is organized as follows. 
In section \ref{sec2}, the convex splitting scheme is constructed, which is linear and unconditionally stable for the AC-LB equation.
We will give some direct energy stability proof by applying the general result of the convex-concave argument.
In section \ref{sec3}, we give a brief review of the classical SDC method and combine the SDC method with the convex splitting method to solve the AC-LB equation. 
{The Fourier pseudo-spectral method} is presented for the spatial discretization in section \ref{sec4}.
Numerical experiments are carried out in section \ref{sec5}, and some concluding remarks will be given in the final section.

\section{Convex splitting method} \label{sec2}

We suppose the domain $\Omega\subset\bbR^d$ is rectangular. 
Let $\Ltwo$  be the space of square-integrable functions. 
The inner product and norm on $\Ltwo$ are denoted by 
\begin{align*}
\IP{\phi}{\psi} := \int_\Omega\phi\psi\DX,
\quad\NLtwo{\phi} := \sqrt{\IP{\phi}{\phi}}.
\end{align*}
For any integer $m>0$, denote
$\Hm{m} := \SET{v\in\Ltwo: D^\xi{v}\in\Ltwo,\SN{\xi}\le{m}},$
where $\xi$ is a non-negative triple index.
Let $\Hmper{m}$ be the subspace composed of periodic functions on $\Hm{m}$.
Define the space $\Unicont$ to consist of all functions which are bounded and uniformly continuous on $\Omega$. 
$\Unicont$ is a Banach space with norm given by $\Nunicont{\phi} := \sup_{\Bx\in\Omega}|\phi(\Bx)|.$
Vector-valued quantities will be denoted by boldface notations, such as $\Ltwov:=(\Ltwo)^d$.

\subsection{A convex splitting of the energy functional}
We introduce a sufficiently large positive constant $S$, the convex splitting form $\Ce(\phi)=\Ce^c(\phi)-\Ce^e(\phi)$ can be taken as
\begin{align}\label{ce-energy}
\begin{split}
\Ce^c(\phi) &= \int_\Omega\left\{\frac{1}2\left((\Delta+1)\phi\right)^2 - \frac\alpha{2!}\phi^2 + \frac{S} 2\phi^2\right\}\DX, \\
\Ce^e(\phi) &= \int_\Omega\left\{\frac{S}2\phi^2 - \frac{1}{4!}\phi^4 + \frac{\gamma}{3!}\phi^3 \right\}\DX,
\end{split}
\end{align}
where “c” (“e”) refers to the contractive (expansive) part of the energy. 
This idea of adding and subtracting a term ${S}/2\NLtwo{\phi}^2$ to a nonlinear energy $\Ce$ to obtain a stable time discretization is based on the convex-concave splitting of Eyre \cite{eyr98}. 

A calculation of the second variation shows
\begin{align*}
\ddv{\Ce^c}{s}(\phi+s\psi)\bigg|_{s=0} &= \int_\Omega\left((\Delta+1)^2 + S - \alpha \right)\psi^2 \DX, \\
\ddv{\Ce^e}{s}(\phi+s\psi)\bigg|_{s=0} &= \int_\Omega\left(S-\frac{\phi^2}{2} + \gamma\phi\right)\psi^2 \DX,
\end{align*}
which implies that $\Ce^c$ is globally convex on $\Ltwo$ if $S>\alpha$, 
and $\Ce^e$ is locally convex depending on $\Nunicont{\phi}$.
Fortunately, we find that the $\Unicont$-bound of the state function $\phi$ is depending on its LB energy. The argument is similar to the scheme analysis in \cite{els13}, but for the sake of completeness, we will provide a condensed version of the proof.

\begin{lemma}\label{lm2.1}
Assume $\alpha < 1 $. For any $\phi \in \Hmper{2}$ with finite energy $\Ce(\phi)$ there is a constant $\lambda>0$ independent of $\phi$ such that 
\begin{align}
\Nunicont{\phi} \le \sqrt{\frac{\Ce(\phi) + (9\gamma^4 + 3)\SN{\Omega}}{\lambda}} := \Cc(\phi)
\end{align}
\end{lemma}

\begin{proof}
By H\"{o}lder and Young inequalities, we deduce that 
\begin{align}\label{use-ineq}
\begin{split}
\frac{\gamma}{3!}\NLp{\phi}{3}^3 &\le \frac1{2} \frac{1}{4!} \NLp{\phi}{4}^4 + 9 \gamma^4 \SN{\Omega},\\ 
\frac{1}{4!} \NLp{\phi}{4}^4 &\ge \NLtwo{\phi}^2 - 6 \SN{\Omega} \\
\NLtwo{\nabla \phi}^2 &= - \SP{\phi}{\Delta \phi} \le \frac 1{2\epsilon} \NLtwo{\phi}^2 + \frac{\epsilon}{2} \NLtwo{\Delta \phi}^2,
\end{split}
\end{align}
for any $\epsilon > 0$.  
By substituting \eqref{use-ineq} into the original energy \eqref{LB-energy}, we get
\begin{align*}
\Ce(\phi) &= \frac{1}{2} \NLtwo{\Delta\phi}^2 - \NLtwo{\nabla \phi}^2 + \frac{1-\alpha}{2}\NLtwo{\phi}^2 - \frac{\gamma}{3!}\NLp{\phi}{3}^3 + \frac{1}{4!}\NLp{\phi}{4}^4 \\
&\ge \frac{1-\epsilon}{2}\NLtwo{\Delta \phi}^2 + \frac{2-\alpha-\frac{1}{\epsilon}}{2} \NLtwo{\phi}^2 - (9 \gamma^4 + 3) \SN{\Omega} \\
&\ge \frac{3}{2} \eta \left( \NLtwo{\Delta \phi}^2 + \NLtwo{\phi}^2  \right) - (9 \gamma^4 + 3) \SN{\Omega} 
\end{align*}
for some $\eta>0$. Then we have 
\begin{align*}
&\Ce(\phi) + (9 \gamma^4 + 3) \SN{\Omega} \ge \eta \left( \NLtwo{\Delta \phi}^2 + \NLtwo{\phi}^2 + \frac{1}{2} \left( \NLtwo{\Delta \phi}^2 + \NLtwo{\phi}^2 \right)  \right) \\
& \ge \eta \left( \NLtwo{\Delta \phi}^2 + \NLtwo{\phi}^2 + \NLtwo{\nabla \phi}^2 \right) = \eta \NHm{\phi}{2}^2 \ge \lambda \Nunicont{\phi}^2
\end{align*}
for some $\lambda >0$ by Sobolev imbedding theorem $\Hmper{2} \subset \Hm{2} \hookrightarrow \Unicont$ (cf. \cite{ada03}). 
The proof is finished.
\end{proof}

For the forthcoming analysis, we introduce the classical convex-concave splitting argument without proof. 
For proof of the lemma, the reader is referred to \cite{wis09}.

\begin{lemma}\label{lm2.2}
Suppose that $\phi,\psi \in \Hmper{2}$, and $\Ce^c,\Ce^e$ are all convex on $\Ltwo$. 
Then 
\begin{align}\label{CS-argue}
\Ce(\phi)-\Ce(\psi) \le \SP{\dphi{\Ce^c}(\phi)-\dphi{\Ce^e}(\psi)}{\phi-\psi},
\end{align}
where $\dphi{}$ are first order variational derivative with respect to $\phi$.
\end{lemma}

\subsection{A linear stable time discretization}\label{sec-CS}
For the choices \eqref{ce-energy}, we obtain a convex splitting scheme of \eqref{AC-LB}: Find $\phi^{n+1} \in \Hmper{2}, n \in \bbN$ such that 
\begin{align}\label{CS}
\frac{\phi^{n+1}-\phi^n}{\Delta t} = - \left(\dphi{\Ce^c}(\phi^{n+1})-\dphi{\Ce^e}(\phi^n)\right) + \beta(\phi^n),
\end{align}
where $\phi^n\approx\phi(t_n)$ is the numerical solution at the $n$-th level $t_n=n\Delta t$ and $\Delta t$ is the step size.

The above scheme satisfies many properties. 
First of all, The scheme \eqref{CS} is explicit in nonlinear terms and hence solves a linear system to generate $\phi^{n+1}$ at the next time level. 
By taking $\Ltwo$ product of \eqref{CS} with $1$ and using integration by parts, we obtain 
\begin{align}\label{mass-CS}
\overline{\phi^{n}} = \overline{{\phi}^0}\quad \forall n \in \bbN,
\end{align}
which implies the proposed scheme can preserve the average mass precisely.
Also, the scheme \eqref{CS} is $\Unicont$-stable and decreases the original energy \eqref{LB-energy} in every step, as shown by the following theorem. 

\begin{theorem}\label{th2.1}
Assume $\alpha < 1 $. For any $\phi^0\in\Hmper{2}$ with finite energy $\Ce(\phi^0)$ there exists a $S>0$ such that the scheme \eqref{CS} is stable for any $\Delta{t}>0$ in the sense
\begin{align}\label{stable-CS}
\Nunicont{\phi^n} \le \Cc^0, \quad \Ce(\phi^{n+1}) \le \Ce(\phi^n) \quad \forall{n}\in\bbN,
\end{align}
where $\Cc^0:= \Cc(\phi^0)$.
\end{theorem}

\begin{proof}
Without loss of generality, we assume $\Cc^0 \ge 1$.
Choose
\begin{align}\label{CS-choose}
S > \max \left( 1, \alpha, \frac{1}{2} (\Cc^0)^2 + \gamma \Cc^0, \frac{\Ce(\Cc^0)+\gamma^2\SN{\Omega}/4}{2\lambda} \right)
\end{align}
We will prove the theorem by induction on $n\in\bbN$, so assume $\Ce(\phi^n) \le \Ce(\phi^0)$ and $\Nunicont{\phi^n} \le \Cc^0$. 

Let us introduce
\begin{align*}
\widetilde{\Ce}^e &= \int_{\Omega} \left\{ \frac{S}{2} \phi^2 - F(\phi) \right\} \DX, \quad \text{where} \\
F(\phi) &= 
\begin{cases}
\frac{1}{4!}\phi^4 - \frac{\gamma}{3!}\phi^3, \quad &\Nunicont{\phi} \le \Cc^0,\\
\left( \frac{1}{4} (\Cc^0)^2 + \frac{\gamma}{2} \Cc^0 \right) \phi^2 + \left( \frac{1}{2}(\Cc^0)^2-\gamma \Cc^0 \right)\SN{\phi},\quad & \text{else}.
\end{cases}
\end{align*}
Since $S$ satisfies \eqref{CS-choose}, a calculation of the second variation yield
\begin{align*}
\ddv{\widetilde\Ce^e}{s}(\phi+s\psi)\bigg|_{s=0} & \ge \int_\Omega\left( S - \left( \frac{1}{2} (\Cc^0)^2 + \gamma \Cc^0 \right) \right)\psi^2 \DX > 0.
\end{align*}
Thus, $\widetilde{\Ce}^e$ is globally convex on $\Ltwo$. 
According to the equivalent argument, we can prove that $\Ce^e$ is convex on $\{\phi \in \Ltwo: \Nunicont{\phi} \le \sqrt{2S} \}$. 
Then, using the convexity of $\Ce^c$ and $\widetilde{\Ce}^e$, we employ the traditional convex-concave splitting argument \eqref{CS-argue},
\begin{align}\label{use-thm-1}
\begin{split}
\Ce^c(\phi^{n+1})-\widetilde{\Ce}^e(\phi^{n+1}) 
&\le \Ce^c(\phi^{n})-\widetilde{\Ce}^e(\phi^{n}) + 
\SP{\delta_\phi \Ce^c(\phi^{n+1}) - \delta_\phi \widetilde{\Ce}^e(\phi^n)}{\phi^{n+1}-\phi^n}\\
&= \Ce^c(\phi^{n})-{\Ce}^e(\phi^{n}) + 
\SP{\delta_\phi \Ce^c(\phi^{n+1}) - \delta_\phi {\Ce}^e(\phi^n)}{\phi^{n+1}-\phi^n}. 
\end{split}
\end{align}
Inserting scheme \eqref{CS} into \eqref{use-thm-1} and using mass conservation \eqref{mass-CS}, we have
\begin{align}\label{use-thm-2}
\Ce^c(\phi^{n+1})-\widetilde{\Ce}^e(\phi^{n+1}) 
\le \Ce(\phi^n) - \frac{1}{\Delta{t}} \NLtwo{\phi^{n+1}-\phi^n}^2 
\le \Ce(\phi^n) \le \Ce(\phi^0).
\end{align}

By same argument as in Lemma \ref{lm2.1}, \eqref{use-thm-2} and \eqref{CS-choose} can lead to a desired bound
\begin{align*}
\Nunicont{\phi^{n+1}} \le \sqrt{\frac{\Ce(\phi^0)+\gamma^2\SN{\Omega}/4}{\lambda}} \le \sqrt{2S}. 
\end{align*}
The proof is finished by using the classical convex-concave splitting argument \eqref{CS-argue} to $\Ce^c-\Ce^e$.
\end{proof}

\begin{remark} 
The linear convex splitting scheme \eqref{CS} can achieve first order convergence by combining stability and first order consistency.
The similar argument is used to analysis error estimate of the convex splitting scheme for the phase field crystal model in \cite{wis09}.
In fact, a careful inspection of its proof shows that it also applies to our cases. 
We do not elaborate on the details.
\end{remark}

\section{Spectral deferred correction method.} \label{sec3}
To construct an efficient numerical method for the AC-LB equation,
we develop a novel SDC method by combining the semi-implicit SDC method \cite{dut00,min03} with the convex splitting method.
First, we present the original SDC method including the classical deferred correction and its some technique details.
The SDC method for the AC-LB equation will be presented next. 

\subsection{The classical deferred correction}
The original deferred correction method was introduced to solve the following Cauchy problem. 
\begin{align}\label{Cauchy-ODE}
\begin{split}
\phi'(t) &= G(\phi) \quad t \in (a,b],\\
\phi(a)&= \phi_a. 
\end{split}
\end{align}
The deferred correction approach works by converting the original ODEs \eqref{Cauchy-ODE} into the corresponding Picard equation
\begin{align}\label{pre-eqn}
\phi(t) = \phi_a + \int_a^t G(\phi(s)) \dmea{s}.
\end{align}
Given an initial approximation $\phi^p$, an error function to measure the approximation is defined by
\begin{align}\label{error}
E(t,{\phi}^p) = \phi_a + \int_a^t G({\phi}^p(s)) \dmea{s} - {\phi}^p(t).
\end{align}
Define correction is $\delta(t) = \phi(t) - {\phi}^p(t)$, 
substituting $\phi(t) = {\phi}^p(t)+\delta(t)$ into \eqref{pre-eqn} and using \eqref{error}, we obtain the correction equation
\begin{align}\label{cor-eqn-1}
\delta(t) = \int_a^t \left( G({\phi}^p(s)+\delta(s)) - G({\phi}^p(s)) \right) \dmea{s} + E(t,{\phi}^p).
\end{align}
After using some numerical method to discretize the correction equation \eqref{cor-eqn-1}, and adding the correction $\delta(t)$ to the initial approximation $\phi^p(t)$, 
we can get higher-order approximated solution $\phi^c(t)$. 
An advantage of this method is that it is a one-step method and can be constructed easily and systematically for any order of accuracy.

\subsection{Subintervals and integral of the interpolant}
The SDC method focuses on a single time interval $[t_n,t_{n+1}]$. 
Given a set of $M$ Gauss-Lobatto quadrature nodes $t_n = \xi_1 < \dots < \xi_M = t_{n+1}$ (cf. \cite{she11}), 
we devide the time interval $[t_n,t_{n+1}]$ into a total of $M-1$ disjoint subintervals, i.e.,
$[t_n,t_{n+1}] = \bigcup_{i=1}^{M-1}[\xi_i,\xi_{i+1}]$.
Let $\Delta \xi_i = \xi_{i+1} - \xi_i$ denote the length of subinterval $[\xi_i,\xi_{i+1}]$.
For convenience, we use the notation $\phi_i = \phi(\xi_i)$.
The same principle also applies to approximations $\delta_i,\phi_i^p,\phi_i^c$.

To compute correction $\delta$ by approximating equation \eqref{cor-eqn-1}, 
the error function $E(t,\phi^p)$ in \eqref{error} must be approximated using numerical quadrature.
Since $\phi^p$ is known at $M$ Gauss-Lobatto quadrature nodes,
we can define the Lagrange interpolation operator $\Ci $ to be the projection onto the space of polynomials of degree at most $M - 1$ via
\begin{align}\label{lag-inter}
\Ci(G(\phi^p))(t) := \sum_{j=1}^M G(\phi^p_j) \ell_j(t),
\end{align}
where $\ell_j(t)$ is Lagrange interpolating basis polynomial corresponding to the spectral point $\xi_j$:
\begin{align*}
\ell_j(t) := \frac{1}{c_j} \prod_{k=1,k \neq j}^M (t-\xi_k), \quad   c_j = \prod_{k=1,k \neq j}^M (\xi_j-\xi_k).
\end{align*}
Then we have the integral of the Lagrange interpolant \eqref{lag-inter} over subinterval $\left[\xi_i,\xi_{i+1}\right]$ as follows.
\begin{align}\label{int-lag}
\int_{\xi_i}^{\xi_{i+1}} \Ci(G(\phi^p))(t) \dmea{t} =\sum_{j=1}^M \omega_{ij} G(\phi^p_j), 
\quad  
\omega_{ij} = \int_{\xi_i}^{\xi_{i+1}} \ell_j(t) \dmea{t},
\end{align}
where $\omega_{ij}$ is quadrature weight.
The coefficients $\omega_{ij}$ can be precomputed, and the quadrature is reduced to a simple matrix-vector multiplication.

\begin{remark}
Given one approximation ${\phi^p}$, the error estimate for the integral of interpolant 
relies on the regularity of solution $\phi$ and the choice of the quadrature rules.
In \cite{cau17}, the author found that Gauss-Lobatto quadrature nodes minimize the error constant and avoid the Runge phenomenon if a large number of quadrature nodes are chosen.
Another advantage of using Gauss-Lobatto nodes is that it contains interval endpoints $t_n$ and $t_{n+1}$, so do not need additional extrapolation.  
\end{remark}

\subsection{SDC method combined with the convex splitting method.}
Suppose we already have the initial numerical approximation $\phi_1^p$ at the left endpoint $\xi_1$.
By section \ref{sec-CS}, a convex splitting method for computing approximation $\phi^p$ to Picard equation \eqref{pre-eqn} is 
\begin{align}\label{pre-CS}
\phi^p_{i+1} = \phi^p_i + \Delta \xi_i \left( G_\text{im}(\phi^p_{i+1}) + G_\text{ex}(\phi^p_i) \right) \quad  i = 1,\dots,M-1,
\end{align}
where $G(\phi) := G_\text{im}(\phi) + G_\text{ex}(\phi)$ and implicit-explicit parts are defined by 
\begin{align}
G_\text{im}(\phi) := - \delta_\phi \Ce^c(\phi), \quad
G_\text{ex}(\phi) := \delta_\phi \Ce^e(\phi) + \beta(\phi). 
\end{align}
Then we focus on the correction process. 
Note that $\phi^c_i = \phi^p_i + \delta_i$.
To be more specific, we set $\phi^c_1 = \phi^p_1$ and $\delta_1 = 0$ as the initial value.
Discretizing the correction equation \eqref{cor-eqn-1} via the convex splitting method, we have
\begin{align*}
\phi^c_{i+1} 
= \phi^c_i + \Delta \xi_i \left( G_\text{im}(\phi^c_{i+1}) + G_\text{ex}(\phi^c_i)
- G_\text{im}(\phi^p_{i+1}) - G_\text{ex}(\phi^p_i)   \right) 
+ \int_{\xi_i}^{\xi_{i+1}} G(\phi^p)(t) \dmea{t}.
\end{align*}
Since the function $G(\phi^p)$ is only known at $M$ Gauss-Lobatto quadrature nodes, 
the last term of the above equation can be computed with the integral of interpolant \eqref{int-lag}. 
Then we get the correction approximation
\begin{align}\label{cor-CS}
\phi^c_{i+1} 
= \phi^c_i + \Delta \xi_i \left( G_\text{im}(\phi^c_{i+1}) + G_\text{ex}(\phi^c_i)
- G_\text{im}(\phi^p_{i+1}) - G_\text{ex}(\phi^p_i) \right) 
+  \int_{\xi_i}^{\xi_{i+1}} \Ci(G(\phi^p)) (t) \dmea{t}.
\end{align}

Iterated deferred correction proceeds by computing a new correction $\phi^c$ to the updated prediction $\phi^p$, and solving the correction equation \eqref{cor-CS} again to obtain a higher order approximation. 
For ease of identification, the SDC method using $M$ Gauss-Lobatto nodes and $K$ correction iterations will be denoted SDC$_M^K$. 
For given initial point approximation $\phi^n$, the SDC$_M^K$ algorithm generates $\phi^{n+1}$ as follows.
\begin{algorithm}[htb]
 \caption{$\phi^{n+1}=$SDC$_M^K(\phi^n)$}
 \begin{algorithmic}\label{alg1}
 \STATE {Set: $\phi_1^p\leftarrow\phi^n$, $\phi_1^c \leftarrow \phi_1^p$.}
 \FOR{{$i = 1:M-1$}}
 \STATE Solve prediction equation \eqref{pre-CS} to get ${\phi_{i+1}^p}$.
 \ENDFOR
 \FOR{$j=1:K$}
 \FOR{$i=1:M-1$}
 \STATE Solve correction equation \eqref{cor-CS} to get $\phi^c_{i+1}$.
 \ENDFOR
 \STATE  {Update the approximate solution: $\phi^p \leftarrow \phi^c$.}
 \ENDFOR
 \STATE {Return: $\phi^{n+1}\leftarrow\phi^c_M$.}
 \end{algorithmic}
\end{algorithm}

Under the assumptions of theorem \ref{th2.1}, 
we find that for  any $\phi^n \in \Hmper{2}$ with finite energy $\Ce(\phi^n)$ there exists a $S>0$ such that the scheme \eqref{pre-CS} is stable for any $\Delta \xi_i>0$ in the sense
\begin{align}\label{stable-pre}
\overline{\phi_{i+1}^p} = \overline{\phi^n},\quad \Nunicont{\phi_{i+1}^p} \le \Cc(\phi^n),\quad \Ce(\phi_{i+1}^p) \le \Ce(\phi_{i}^p)\quad i = 1,\dots, M-1.
\end{align}

By the definition of the operator $\Ci$ and using integration by parts, we have
\begin{align*}
\int_{\Omega} \int_{\xi_i}^{\xi_{i+1}}  \Ci(G(\phi^p))(t)  \dmea{t} \DX = 0.
\end{align*}
Taking $\Ltwo$ product of \eqref{cor-CS} with 1, and combining the above identity with \eqref{stable-pre} yield  
\begin{align}
\overline{\phi_{i+1}^c} = \overline{\phi^n}\quad i = 1,\dots,M-1,
\end{align}
which implies SDC method combined with the convex splitting method can preserve the mass precisely. 
 
\begin{remark}
The global order of accuracy for SDC$_M^K$ method is $\min\{M,K\}$ where $K$ is the number of correction iterations \cite{cau17,dut00}.
The stability properties of the SDC methods are examined in \cite{dut00,min03}.
In \cite{cau17}, the authors' proofs pointed to a total of three sources of errors that SDC methods carry: the error at the current time point, the error from the previous iterate, and the numerical integration error that comes from the total number of quadrature nodes used for integration.
It was found that the accuracy improvement of SDC method may affect the overall energy stability which is intrinsic to the phase field models \cite{fen15}.
\end{remark}

\subsection{Adaptive correction strategy.}

The correction number for SDC$_M^K$ algorithm is $K(M-1)$.
Using too many corrections costs too much time. 
Without using multiple corrections, accuracy may not be satisfactory. 
It reminds us to balance accuracy and stability.
Inspired by adaptive restart technique \cite{odo15} for accelerated gradient schemes, 
we provide an adaptive-SDC (ASDC$_M^K$) algorithm that makes some computationally cheap observation and decides whether or not to use correction based on that observation.

\begin{algorithm}[htb!]\label{alg2}
 \caption{$\phi^{n+1}=$ASDC$_M^K(\phi^n)$}
 \begin{algorithmic}
 \STATE {Set: $\phi_1^p\leftarrow\phi^n$, $\phi_1^c \leftarrow \phi_1^p$.}
 \FOR{{$i = 1:M-1$}}
 \STATE Solve prediction equation \eqref{pre-CS} to get ${\phi_{i+1}^p}$.
 \ENDFOR
 \STATE {$ k \leftarrow 1 $} 
 \FOR{$j=1:K$}
 \FOR{$i=k:M-1$}
 \STATE Solve correction equation \eqref{cor-CS} to get $\phi^c_{i+1}$.
 \IF{$\SP{\delta_\phi\Ce^c(\phi^c_{i+1})-\delta_\phi\Ce^e({\phi^c_i})}{\phi_{i+1}^c-\phi_i^c}<0$}
 \STATE {$ \phi_{i}^c \leftarrow \phi_{i+1}^c $ }
 \STATE {$ k \leftarrow i $ }
 \ENDIF
 \ENDFOR
 \STATE  {Update the approximate solution: $\phi^p \leftarrow \phi^c$.}
 \ENDFOR
 \STATE {Return: $\phi^{n+1}\leftarrow\phi^c_M$.}
 \end{algorithmic}
\end{algorithm}

Note that we use a convex-concave argument to control the correction number. 
Because variable index $k$ satisfies $k < {M}$, we don't change the correction times of the approximation on the $M$-th Gauss-Lobatto point.
If the the argument $\SP{\delta_\phi\Ce^c(\phi^c_{i+1})-\delta_\phi\Ce^e({\phi^c_i})}{\phi_{i+1}^c-\phi_i^c}<0$ satisfies, 
using the convex-concave argument \eqref{CS-argue} yields $\Ce(\phi_{i+1}^c) \le \Ce(\phi_{i}^c).$ 
Frome \eqref{stable-pre}, we find that 
\begin{align*}
\Ce(\phi_{i+1}^c) \le \Ce(\phi^n).
\end{align*}
By the same argument in Lemma \ref{lm2.1}, we have the following $\Unicont$-bound of $\phi_{i+1}^c$:
\begin{align*} 
\Nunicont{\phi_{i+1}^c} \le \Cc(\phi^n). 
\end{align*}
We can reset the $ i$-th approximation by $\phi_{i+1}^c$ and adjust the initial position $ k $  for the next correction loop. 
In other words, the energy-decreasing property will be preserved and fewer corrections are required for stable solutions. 
Otherwise, using more corrections generates approximations until we observe stable solutions.
Therefore, we get a mass conservative and energy-stable spectral deferred correction method for the AC-LB equation.

\section{Spatial discretization}\label{sec4}
The purpose of this section is to construct 
Fourier pseudo-spectral method 
for the LB model.
Without loss of generality, we denote the rectangular domain $ \Omega $  as $\Pi_{j=1}^d[0,L_j] \subset \bbR^d$. 
For a given positive even integer $ N $, we denote discretized gridpoint space as
\[ \Cp = \SET{ \Bx:=(x_1,x_2,\dots,x_d)\in\bbR^d: x_j = nL_j/N, n=1,2,\dots,N, j=1,2,\dots,d }.\] 
Define the discrete Fourier spectral space 
\[ \Ck = \SET{ \Bk:=(k_1,k_2,\dots,k_d)\in\bbZ^d: \SN{k_j} \le N/2, j=1,2,\dots,d}. \]
Next, we consider the discrete Fourier transform (DFT) of $\phi$ on the gridpoint $\Bx\in\Cp$ and its inverse as
\begin{align}\label{DFT}
\wh{\phi}(\Bk) =   \sum_{\Bx\in\Cp} \phi(\Bx) e^{-i2\pi(\BB\Bk)\cdot\Bx},\quad 
\phi(\Bx) = \frac{1}{N^d}  \sum_{\Bk\in\Ck} \wh{\phi}(\Bk) e^{i2\pi (\BB\Bk)\cdot\Bx},
\end{align}
where $\BB = \diag{L_1^{-1},L_2^{-1},\dots,L_d^{-1}}\in\bbR^{d \times d}$ is the scaling matrix.
We assume that $ \phi(\Bx)  $  is sufficiently smooth. 
The following equation shows that the differentiation works with DFT. 
\begin{align*}
\frac{\partial}{\partial x_j} \phi(\Bx) = \sum_{k\in\Ck} \left(i \frac{2\pi k_j}{L_j}\right)  \wh{\phi}(\Bk) e^{i2\pi(\BB\Bk)\cdot\Bx} .  
\end{align*}
Therefore we can represent the Laplacian to coefficients in the discrete Fourier space as follow:
\begin{align*}
\Delta \phi(\Bx) = \sum_{k\in\Ck} -2\pi\SN{\BB\Bk}^2 \wh{\phi}(\Bk) e^{i2\pi(\BB\Bk)\cdot\Bx}, 
\end{align*}
where $ \SN{\BB\Bk}:=(\BB\Bk\cdot\BB\Bk)^{1/2} $ is the Euclidean norm of $ \BB\Bk \in \bbR^d$. 
Thus we transform convex splitting scheme \eqref{CS} into discrete Fourier space as follows:
For the given data $\phi^n(\Bx),\forall\Bx\in\Cp$, 
find $\wh{\phi^{n+1}}(\Bk),\forall\Bk\in\Ck$ such that 
\begin{align*}
\wh{\phi^{n+1}}(\Bk) = \frac{\wh{\phi^n}(\Bk)+\Delta{t}\left(\wh{\delta_\phi\Ce^e({\phi^n})}(\Bk) + \beta(\phi^n) \right)}{1+\Delta{t}\left(S-\alpha+(1-2\pi\SN{\BB\Bk}^2)\right)}.  
\end{align*}
A direct evaluation of the nonlinear term $\wh{\delta_\phi\Ce^e(\phi^n)}(\Bk)$ is extremely expensive. 
Thanks that ${\delta_\phi\Ce^e(\phi^n)}$ is a simple multiplication in the $d$-dimensional real space. 
The pseudo-spectral method takes the advantage of this observation by evaluating   $\wh{\delta_\phi\Ce^e(\phi^n)}(\Bk)$ in the real space via the Fast Fourier Transformation algorithm. 
As a result, it provides an efficient technique to reduce the computation cost.
Then, the updated numerical solution's data $ \phi^{n+1}(\Bx),\forall\Bx\in\Cp $ can be computed using inverse DFT in \eqref{DFT}.  
Similarly, approximations \eqref{pre-CS} and \eqref{cor-CS} can also be transformed into discrete Fourier space. 

Discretize the energy functional of LB model \eqref{LB-energy} with spectral derivative, it yields
\begin{align}\label{DFT-energy}
\Ce_\Cp(\phi) = \frac{\SN{\Omega}}{N^{2d}}  \sum_{\Bk\in\Ck} \frac{(1-2\pi\SN{\BB\Bk}^2)^2-\alpha}{2} \SN{\wh{\phi}(\Bk)}^2 
+ \frac{\SN{\Omega}}{N^d} \sum_{\Bx\in\Cp} \left(\frac{1}{4!} \phi(\Bx)^4 - \frac{\gamma}{3!} \phi(\Bx)^3\right), 
\end{align}
since the following discrete Parseval’s identity can be applied,
\begin{align}\label{parseval}
\sum_{\Bx\in\Cp} \phi(\Bx)^2 = \frac{1}{N^d} \sum_{\Bk\in\Ck} \SN{\wh{\phi}(\Bk)}^2.
\end{align}
The computation of the DFT and spectral derivatives can be accomplished by the Fast Fourier Transform to reduce floating point operations.

\section{Numerical experiments} \label{sec5}
Now we carry out the numerical experiments for the LB model to demonstrate the performance of the proposed method.
All experiments were performed on a workstation with a 2.90 GHz CPU (intel Xeon Gold 6326, 16 processors). 
All codes were written in MATLAB language without parallel implementation.

\begin{example}\label{ex1}
We first examine the convergence of the SDC method. 
The problem's setting is the following:
\begin{equation*}
\Omega=\left[0,{16\pi}/{\sqrt{3}}\right]\times\left[0,8\pi\right],\quad \alpha=0.15,\quad \gamma=0.25,\quad 
  S=2.
\end{equation*}
To verify the convergence rate, we add a source term to the AC-LB equation such that the exact solution is
\begin{equation*}
\phi(t,x,y)=e^{-2t}\sin{\sqrt{3}x}\sin{y}
\end{equation*}
We use the DFT in Section \ref{sec4} for the spatial discretization and SDC$_M^K$ algorithm with Legendre-Gauss-Lobatto quadrature points for the time discretization.
We compute approximation $\phi_\Cp$ on a spatial grid $\Cp$ with $N=512$ and take a rough step size $\Delta{t}=0.05$.
Table \ref{tab1} shows the $ \Ltwo $ norm error and $ \Unicont $ norm error at $T=4$ with $M=4$ and $K=1,2,3,4$ respectively.
Optimal convergence rates are obtained for the SDC method concerning the step size.

\begin{table}[!htbp]
\centering
\caption{
Convergence rates against the number of corrections $K$ in SDC$_M^K$ algorithm.
}
\label{tab1}
\vspace{10pt}
\begin{tabular}{|c|c|c|c|c|c|}
\hline
           Method                     &                      Error & $\Delta{t}=$0.05 & $\Delta{t}/2$ & $\Delta{t}/4$ & $\Delta{t}/8$ \\ \hline
           \multirow{4}{*}{SDC$_4^1$} &  $\NLtwo{\phi_\Cp - \phi}$ & 1.7949e-05       &    5.4847e-06 & 1.5304e-06    &    4.0540e-07 \\
                                      &                      Order & --               &        1.7104 & 1.8415        &        1.9165 \\
                                      & $\Nunicont{\phi_\Cp-\phi}$ & 1.3292e-06       &    4.0618e-07 & 1.1334e-07    &    3.0022e-08 \\
                                      &                      Order & 0                &        1.7104 & 1.8415        &        1.9166 \\ \hline
           \multirow{4}{*}{SDC$_4^2$} &  $\NLtwo{\phi_\Cp - \phi}$ & 1.5222e-06       &    2.8287e-07 & 4.3481e-08    &    6.0440e-09 \\
                                      &                      Order & --               &        2.4279 & 2.7017        &        2.8468 \\
                                      & $\Nunicont{\phi_\Cp-\phi}$ & 1.1273e-07       &    2.0948e-08 & 3.2201e-09    &    4.4759e-10 \\
                                      &                      Order & --               &        2.4280 & 2.7016        &        2.8469 \\ \hline
           \multirow{4}{*}{SDC$_4^3$} &  $\NLtwo{\phi_\Cp - \phi}$ & 1.2966e-07       &    1.4946e-08 & 1.2704e-09    &    9.2766e-11 \\
                                      &                      Order & --               &        3.1169 & 3.5564        &        3.7755 \\
                                      & $\Nunicont{\phi_\Cp-\phi}$ & 9.6019e-09       &    1.1069e-09 & 9.4084e-11    &    6.8706e-12 \\
                                      &                      Order & --               &        3.1168 & 3.5564        &        3.7754 \\ \hline
           \multirow{4}{*}{SDC$_4^4$} &  $\NLtwo{\phi_\Cp - \phi}$ & 1.0856e-08       &    8.0012e-10 & 3.7792e-11    &    1.4514e-12 \\
                                      &                      Order & --               &        3.7621 & 4.4041        &        4.7026 \\
                                      & $\Nunicont{\phi_\Cp-\phi}$ & 8.0395e-10       &    5.9255e-11 & 2.7994e-12    &    1.0810e-13 \\
                                      &                      Order & --               &        3.7621 & 4.4037        &        4.6947 \\ \hline
\end{tabular}
\end{table}
\end{example}

\begin{example}\label{ex2}
This example is to verify the mass conservative and energy stable properties of the SDC method. 
We use two-dimensional periodic crystals of lamellar phase and cylindrical phase,
as show in Figure \ref{fig1}, to demonstrate the performance of SDC$ _M^K $ algorithm.
The computational domain is $ \Omega=\left[0,{16\pi}/{\sqrt{3}} \right]\times\left[0,8\pi\right] \subset \bbR^2$. 
The parameters of LB model \eqref{LB-energy} are set as $ \alpha=0.15,\gamma=0.25 $.
The initial approximation to those phases which can be found in \cite{shi03} is chosen as  
\begin{align*}
\phi^0(\Bx) = 2 a_1 \cos(\BG_1\cdot\Bx) + 2 a_2 \left(\cos(\BG_2\cdot\Bx) + \cos(\BG_3\cdot\Bx)\right)\quad \forall\Bx\in\Cp,
\end{align*}
where the $\BG_i$ are given by
\begin{align*}
\BG_1=(0,1),\quad \BG_2 =  (-\sqrt{3}/2,1/2),\quad \BG_3 =  (-\sqrt{3}/2,-1/2).
\end{align*}
The lamellar phase is described by $a_1=\sqrt{2\alpha},a_2=0$, 
and the cylindrical phase is described by $a_1= a_2=(\gamma+\sqrt{\gamma^2+10\alpha})/5$.
Note that the initial phases satisfy $ \overline{\phi^0}=0 $. 

The spatial grip $\Cp$ is fixed with $N=512$. 
We take $M=4$ and change $K=1,2,3,4$ to implement SDC$_M^K$ algorithm.
We choose the positive constant $S=2$ to allow a pretty large step size $\Delta{t}=1$.
To show the energy dissipation obviously, we calculate a reference energy $\Ce_s$ by choosing the
invariant energy value as the grid size converges to $0$.
From our numerical tests, the reference energy has $14$ significant decimal digits.
The reference energy values of lamellar and cylindrical phases are $-16.532074091947$ and $-17.324103376071$ respectively.
Figure \ref{fig1} (a) and (b) show the stationary solutions of the lamellar and cylindrical phases respectively.

\begin{figure}[!htbp]
\centering
\subfigure[Lamellar phase]{
\includegraphics[width=2.5in]{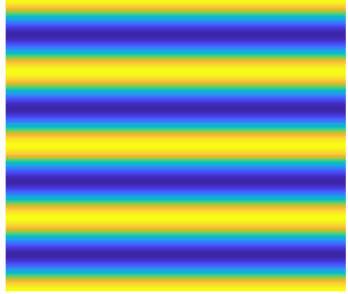}}
\subfigure[Cylindrical phase]{
\includegraphics[width=2.5in]{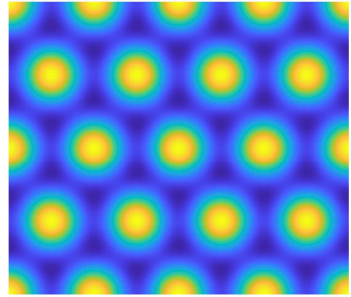}}
\caption{The two-dimensional periodic crystals in LB model with $ \alpha = 0.15, \gamma = 0.25 $.}
\label{fig1}
\end{figure}

Figure \ref{fig2} gives iteration process of SDC$_M^K$ algorithm for the lamellar phase, 
including the energy difference and the average mass during iterations.
We do numerical experiments when Legendre-Gauss-Lobatto, 
and Chebyshev-Gauss-Lobatto quadrature nodes are used to construct the polynomial interpolant \eqref{lag-inter}.
It is observed that SDC$_M^K$ algorithm has energy dissipative and mass conservative properties no matter what kind of quadrature points we use.  
The numerical behavior of SDC$_M^K$ algorithm for the cylindrical phase can be found in Figure \ref{fig3}.
We find again that our proposed approaches are mass conservative and energy stable.

For SDC$_M^K$ algorithm, an obvious observation is that the rate of energy difference descent doesn't change when correction number $K$ increases.
Therefore, the balance between efficiency and accuracy should be considered when using the SDC method.
In addition, we find that the mass curve was slightly disturbed during the iteration.
\begin{figure}[!htbp]
\centering
\subfigure[Energy difference during iterations]{
\includegraphics[width=2.5in]{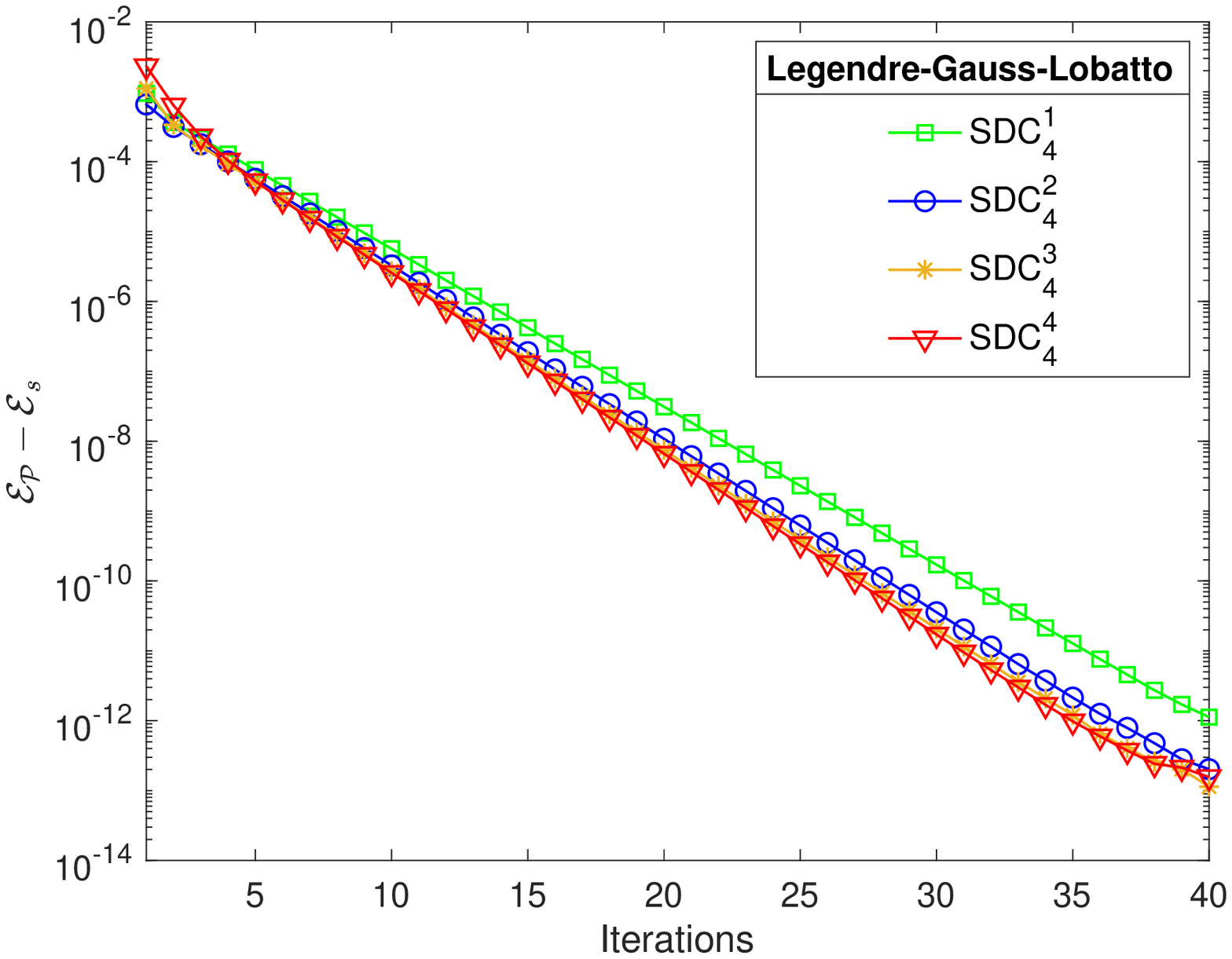}}
\subfigure[Average mass during iterations]{
\includegraphics[width=2.5in]{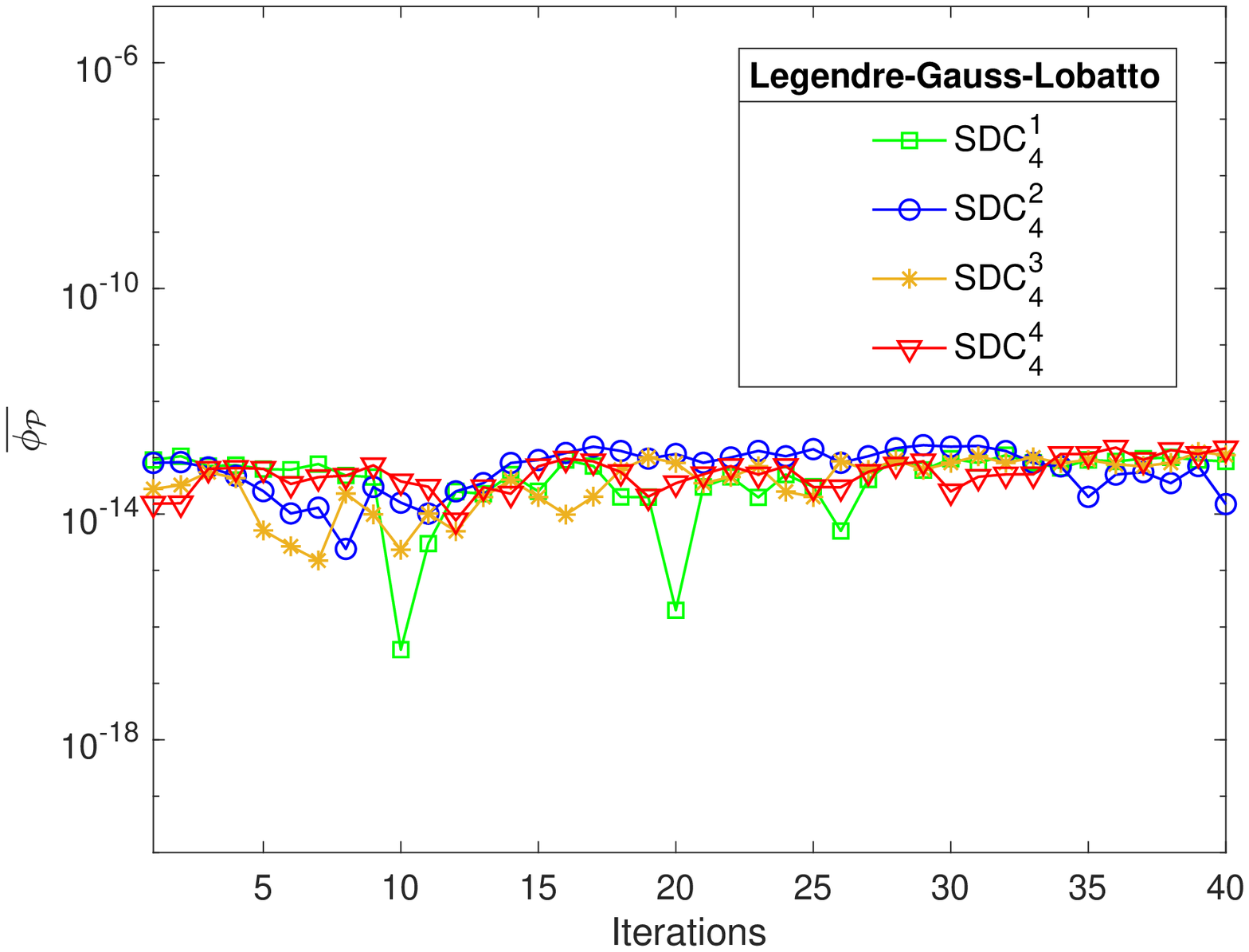}}
\subfigure[Energy difference during iterations]{
\includegraphics[width=2.5in]{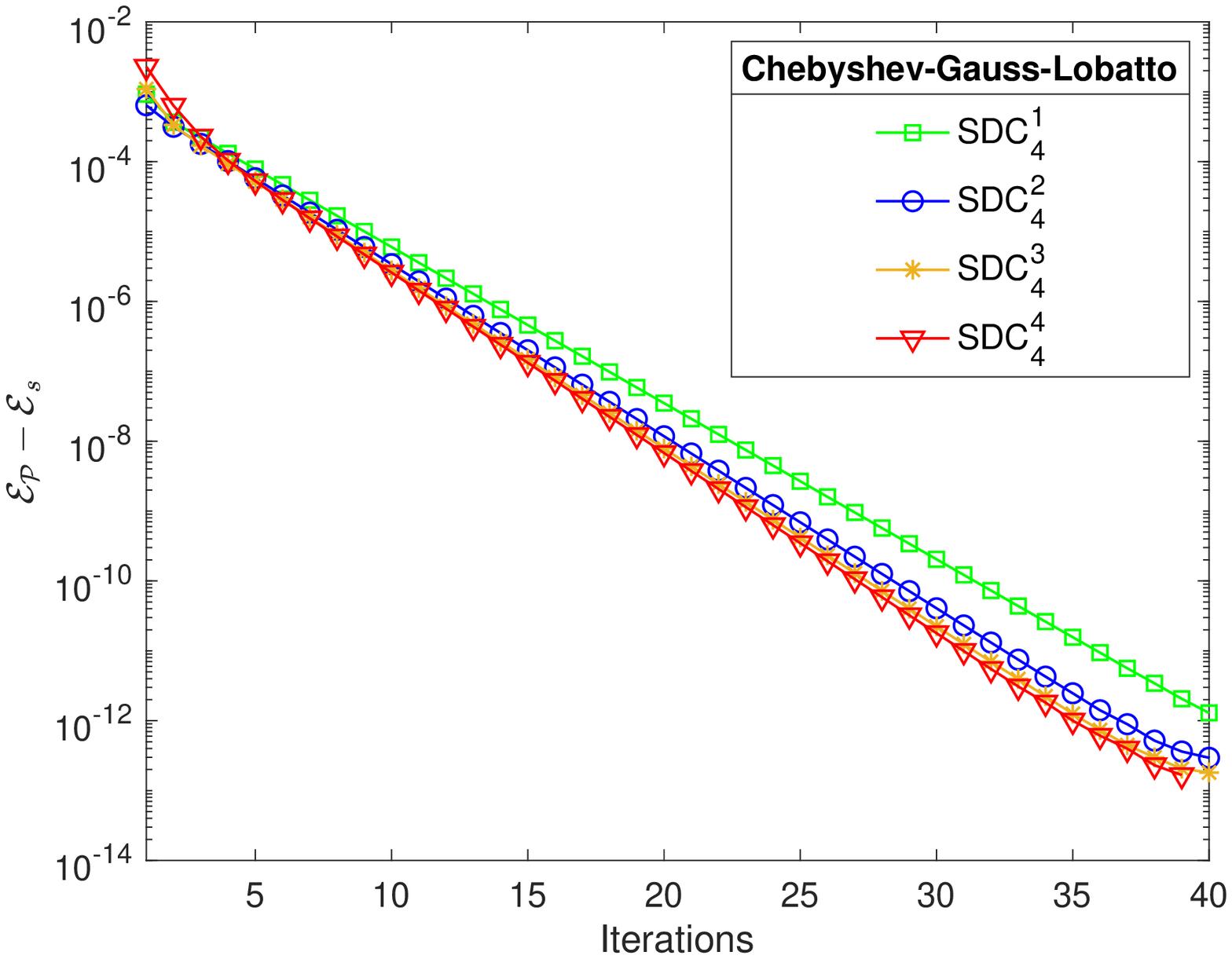}}
\subfigure[Average mass during iterations]{
\includegraphics[width=2.5in]{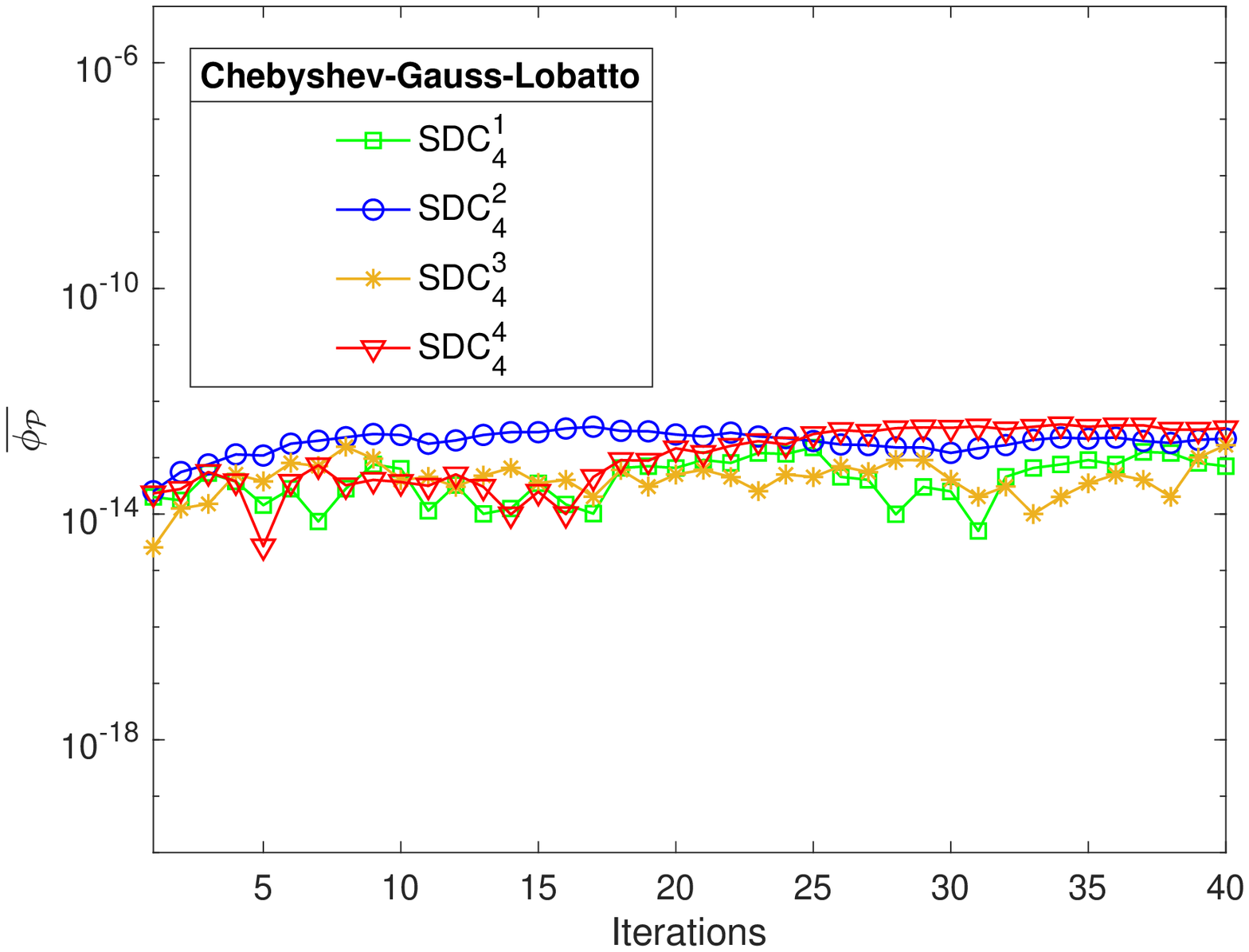}}
\caption{Lamellar phase: Numerical behavior of SDC$_M^K$ algorithm with $\Delta{t}=1$; 
\textit{First row:} Legendre-Gauss-Lobatto quadrature nodes; \textit{Second row:} Chebyshev-Gauss-Lobatto quadrature nodes.} 
\label{fig2}
\end{figure}

\begin{figure}[!htbp]
\centering
\subfigure[Energy difference during iterations]{
\includegraphics[width=2.5in]{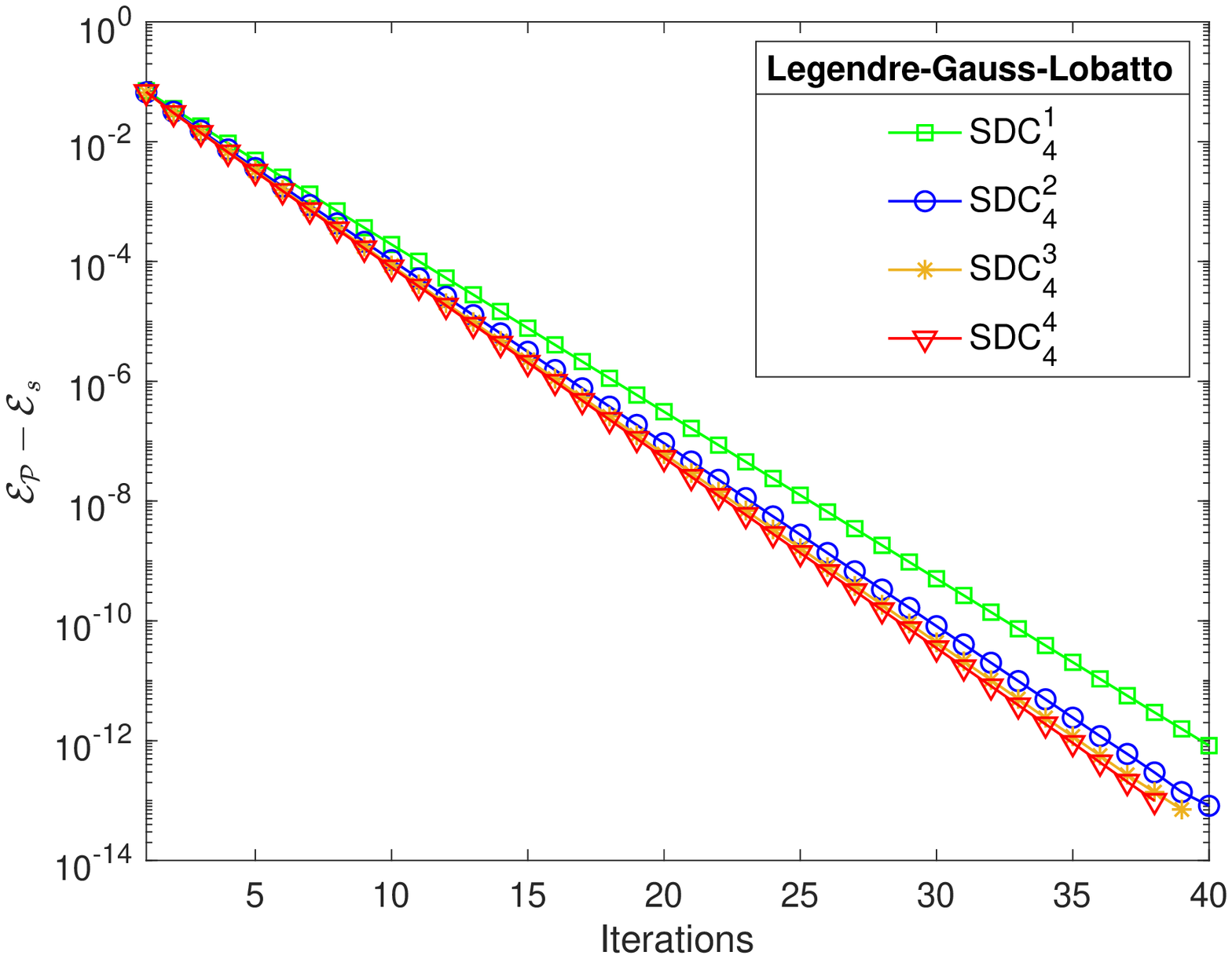}}
\subfigure[Average mass during iterations]{
\includegraphics[width=2.5in]{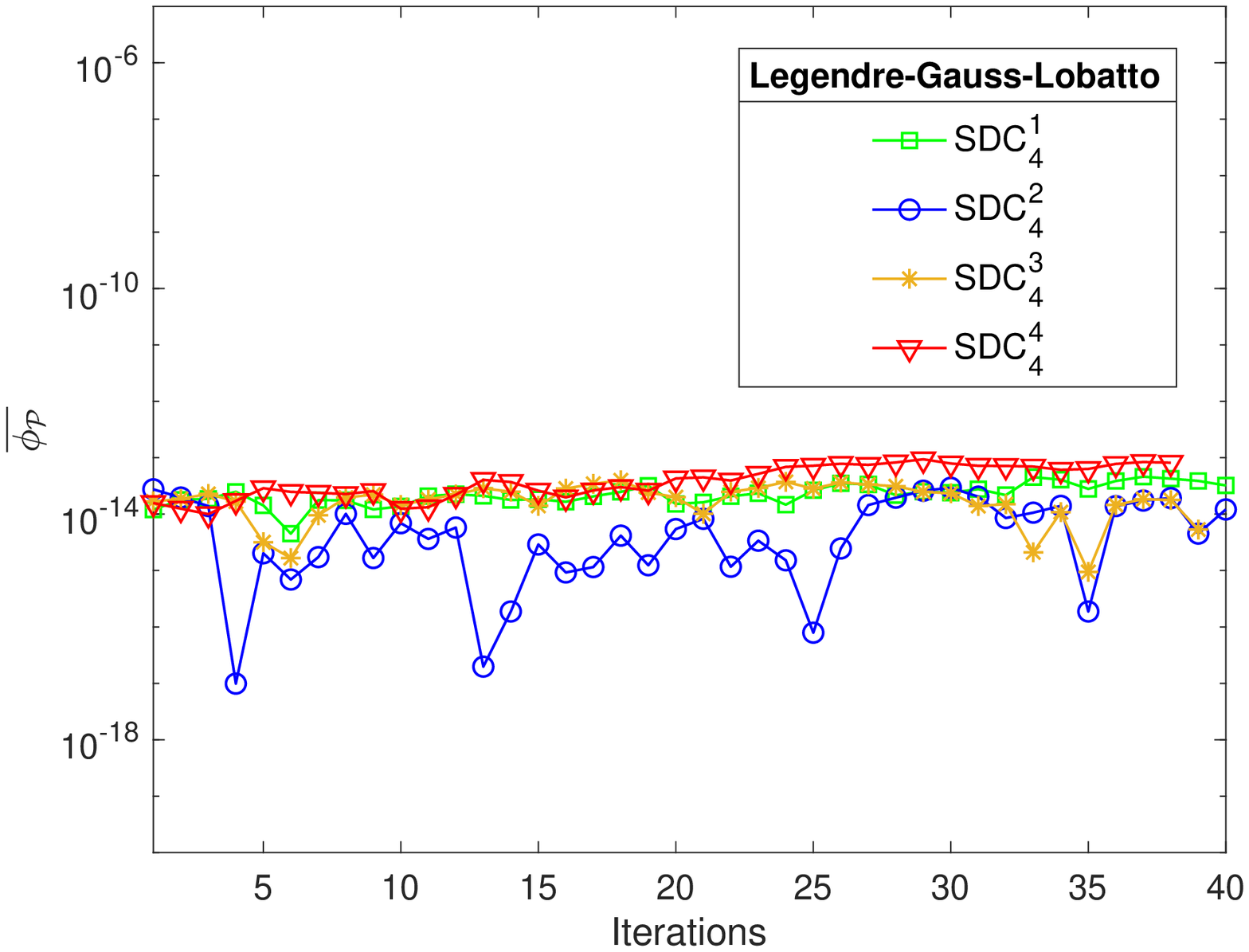}}
\subfigure[Energy difference during iterations]{
\includegraphics[width=2.5in]{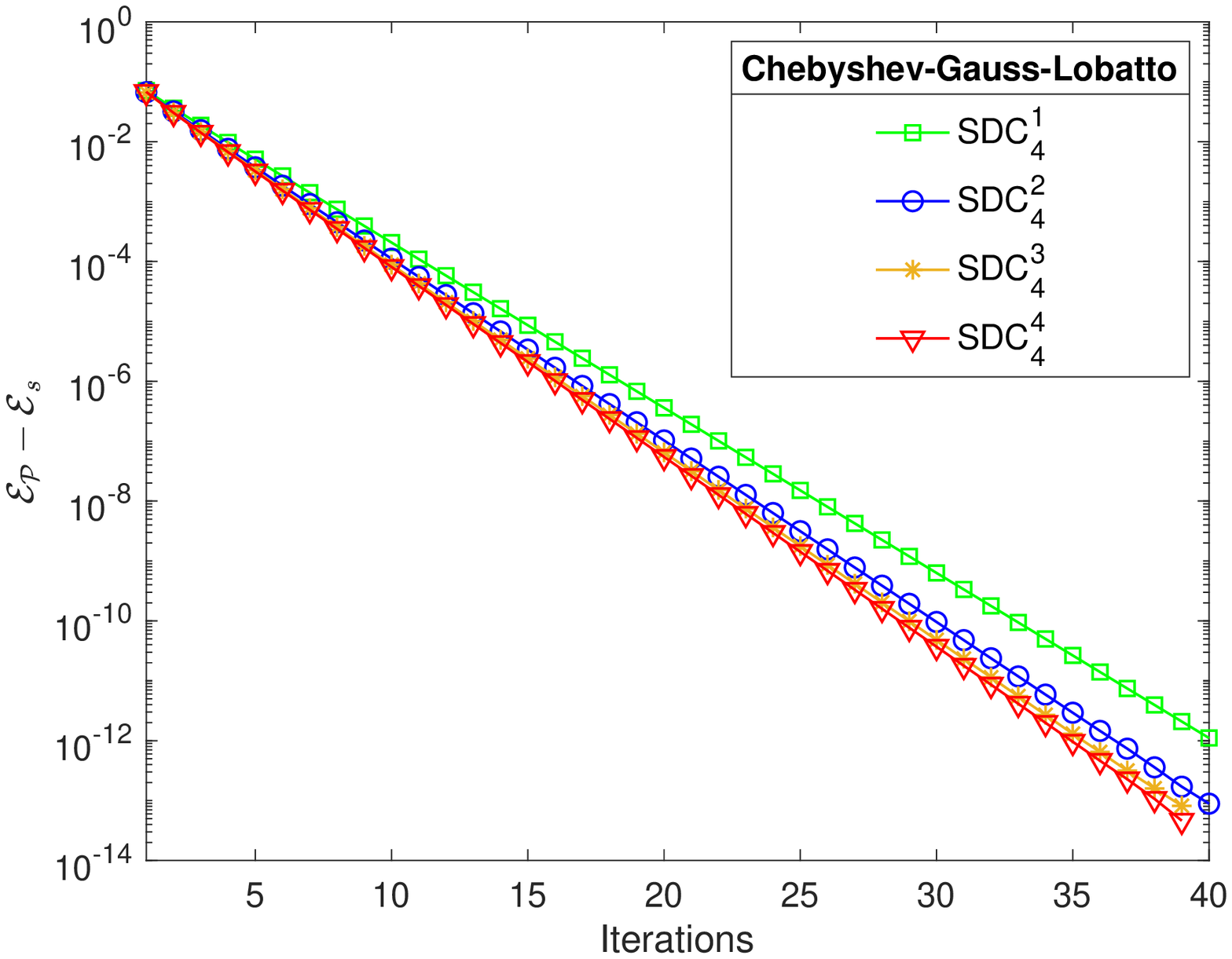}}
\subfigure[Average mass during iterations]{
\includegraphics[width=2.5in]{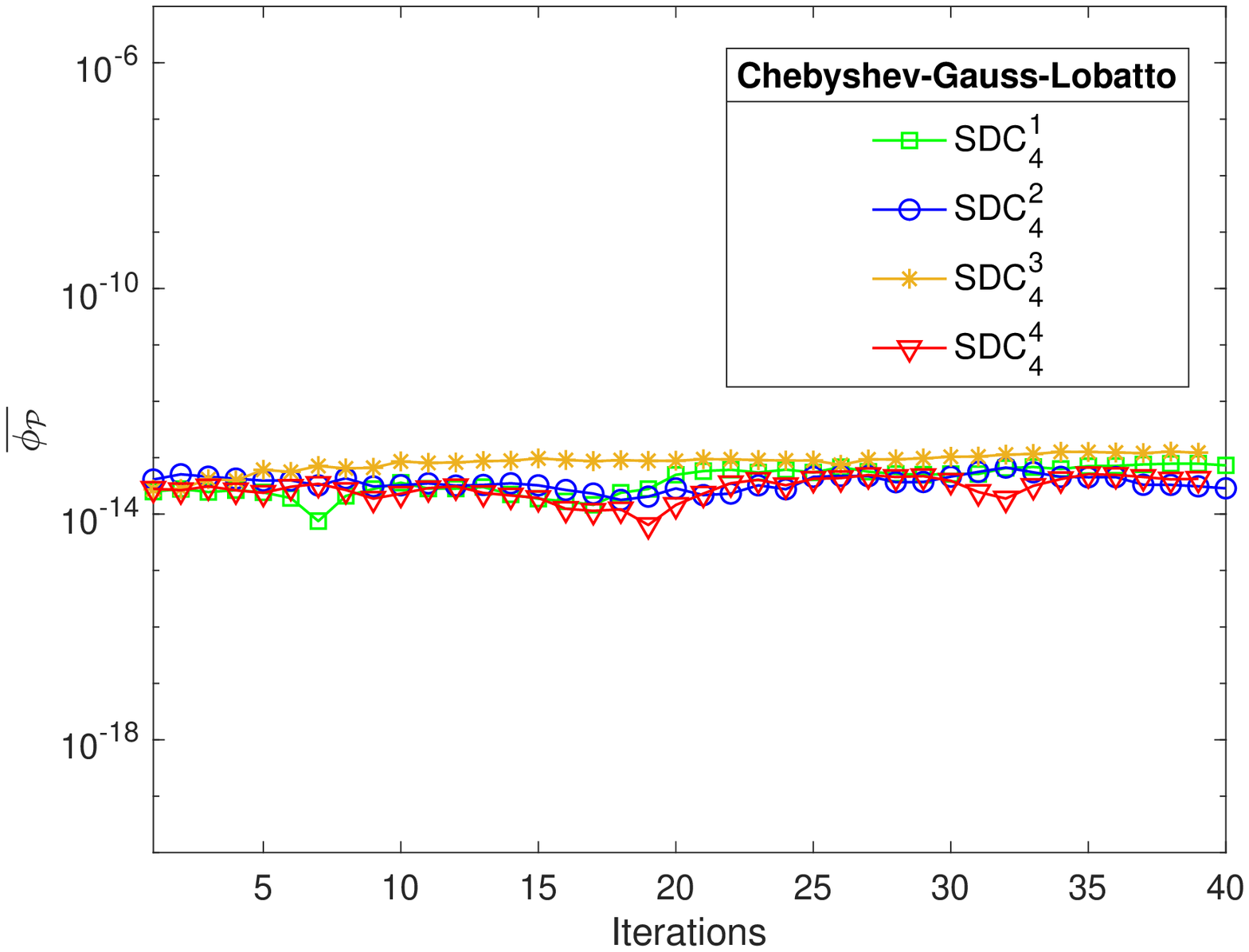}}
\caption{Cylindrical phase: Numerical behavior of SDC$_M^K$ algorithm with $\Delta{t}=1$; 
\textit{First row:} Legendre-Gauss-Lobatto quadrature nodes; \textit{Second row:} Chebyshev-Gauss-Lobatto quadrature nodes.} 
\label{fig3}
\end{figure}
\end{example}

\begin{example}\label{ex3}
The purpose of this example is to investigate the performance of ASDC$_M^K$ algorithm by the two-dimensional periodic crystals in Example \ref{ex2}.
The setting for this example is same as that for Example 2 excluding the initial correction number $K$. 
For computing a stationary solution, we stop iteration when the following criteria is met:
\begin{equation}\label{stop}
  \Ce_\Cp - \Ce_s \le \varepsilon,
 \end{equation}
where $\varepsilon>0$.

We set $\varepsilon=10^{-12}$ and change $K=2,3,4,5$ to implement ASDC$_M^K$ algorithm.
To compare the efficiency with SDC$_M^K$ algorithm, 
let $ N_\text{correction} $ denote the average number of times the correction equation \eqref{cor-CS} was solved.
For computing the lamellar phase, 
Table \ref{tab2} shows the average correction number $ N_\text{correction} $, the total iteration number $ N_\text{iteration} $, and the final energy difference $\Ce_p-\Ce_s$.
Clearly, for ASDC$_M^K$ algorithm both $ N_\text{iteration} $ and $ N_\text{correction} $ are less than SDC$_M^K$ algorithm whenever Legendre-Gauss-Lobatto 
or Chebyshev-Gauss-Lobatto quadrature points are used.
Moreover, $ N_\text{iteration} $ of ASDC$_M^K$ algorithm decreases with the increase of the initial correction number $K$, 
while $ N_\text{iteration} $ of SDC$_M^K$ increases. 
Especially, when $K=5$, ASDC$_M^K$ algorithm only needs less than half of the iterations of SDC$_M^K$ algorithm.
Therefore, ASDC$_M^K$ algorithm is more efficient than SDC$_M^K$ algorithm.
Again, as shown in Table \ref{tab3}, ASDC$_M^K$ algorithm demonstrates a better performance over SDC$_M^K$ algorithm in computing the cylindrical phase.

\begin{table}[!htbp]
\centering
\caption{Numerical results of SDC$_M^K$ and ASDC$_M^K$ algorithms with $\Delta{t}=1,M=4$ for computing the lamellar phase.}
\label{tab2}
\vspace{10pt}
\begin{tabular}{|c|c|c|c|c|}
\hline

\multirow{3}{*}{$K$} & \multicolumn{4}{c|}{$N_\text{iteration}(N_\text{correction})$} \\ \cline{2-5}

  & \multicolumn{2}{c|}{Legendre-Gauss-Lobatto} & \multicolumn{2}{c|}{Chebyshev-Gauss-Lobatto} \\ \cline{2-5}
  &                                   SDC$_M^K$ & ASDC$_M^K$                                    &   SDC$_M^K$ & ASDC$_M^K$           \\ \hline
2 &                                      37 (6) & 32 (5)                                        &      37 (6) & 33 (5)               \\ \hline
3 &                                      36 (9) & 27 (7)                                        &      36 (9) & 28 (7)               \\ \hline
4 &                                 \;\;35 (12) & 23 (9)                                        & \;\;35 (12) & 25 (9)               \\ \hline
5 &                                 \;\;41 (15) & \;\;\textbf{21} (11)                          & \;\;45 (15) & \;\;\textbf{22} (11) \\ \hline

\multirow{3}{*}{$K$}& \multicolumn{4}{c|}{${\Ce_\Cp-\Ce_s}$} \\ \cline{2-5}

  & \multicolumn{2}{c|}{Legendre-Gauss-Lobatto} & \multicolumn{2}{c|}{Chebyshev-Gauss-Lobatto} \\ \cline{2-5}
  &                                   SDC$_M^K$ & ASDC$_M^K$                                    &   SDC$_M^K$ & ASDC$_M^K$ \\ \hline
2 &                                  9.0239e-13 & 9.5923e-14                                    &  7.5673e-13 & 7.1054e-14 \\ \hline
3 &                                  6.6080e-13 & 9.5923e-14                                    &  8.2778e-13 & 8.1712e-14 \\ \hline
4 &                                  7.5673e-13 & 5.6843e-14                                    &  9.8055e-13 & 9.9476e-14 \\ \hline
5 &                                  8.1357e-13 & 6.0396e-14                                    &  8.2423e-13 & 7.1054e-14 \\ \hline

\end{tabular}
\end{table}

\begin{table}[!htbp]
\centering
\caption{Numerical results of SDC$_M^K$ and ASDC$_M^K$ algorithms with $\Delta{t}=1,M=4$ for computing the cylindrical phase.}
\label{tab3}
\vspace{10pt}
\begin{tabular}{|c|c|c|c|c|}
\hline

 \multirow{3}{*}{$K$} & \multicolumn{4}{c|}{$N_\text{iteration}(N_\text{correction})$} \\ \cline{2-5}
 
  & \multicolumn{2}{c|}{Legendre-Gauss-Lobatto} & \multicolumn{2}{c|}{Chebyshev-Gauss-Lobatto} \\ \cline{2-5}
  &                                   SDC$_M^K$ & ASDC$_M^K$                                    &   SDC$_M^K$ & ASDC$_M^K$           \\ \hline
2 &                                      37 (6) & 32 (5)                                        &      37 (6) & 33 (5)               \\ \hline
3 &                                      36 (9) & 27 (7)                                        &      36 (9) & 28 (7)               \\ \hline
4 &                                 \;\;35 (12) & 23 (9)                                        & \;\;35 (12) & 24 (9)               \\ \hline
5 &                                 \;\;38 (15) & \;\;\textbf{21} (11)                          & \;\;41 (15) & \;\;\textbf{21} (11) \\ \hline

\multirow{3}{*}{$K$}& \multicolumn{4}{c|}{${\Ce_\Cp-\Ce_s}$} \\ \cline{2-5}

  & \multicolumn{2}{c|}{Legendre-Gauss-Lobatto} & \multicolumn{2}{c|}{Chebyshev-Gauss-Lobatto} \\ \cline{2-5}
  &                                   SDC$_M^K$ & ASDC$_M^K$                                    &   SDC$_M^K$ & ASDC$_M^K$ \\ \hline
2 &                                  8.1712e-14 & 9.5923e-14                                    &  8.5265e-14 & 5.3291e-14 \\ \hline
3 &                                  6.7502e-14 & 9.5923e-14                                    &  7.8160e-14 & 7.4607e-14 \\ \hline
4 &                                  4.9738e-14 & 5.6943e-14                                    &  4.9738e-14 & 7.8160e-14 \\ \hline
5 &                                  7.4607e-14 & 6.0396e-14                                    &  8.5265e-14 & 8.8818e-14 \\ \hline
\end{tabular}
\end{table}
\end{example}

\begin{example}\label{ex4}
In this example, we use three-dimensional periodic crystals of the A15 phase, the body-centered cubic (BCC) phase, the face-centered-cubic (FCC) phase, and the double gyroid (GYR) phase,
to test the robustness of the parameters in LB model \eqref{LB-energy}.
The A15 phase is a cubic phase with two nonequivalent types of lattice sites: one whose atoms sit at the edges and center of the conventional unit cell, and one whose atoms are placed along lines subdividing the cubic faces into two congruent parts \cite{sin72}.
The BCC phase has one lattice point in the center of the unit cell in addition to the eight corner points \cite{wol85}.
The FCC phase has lattice points on the faces of the cube, each giving exactly one-half contribution, in addition to the corner lattice points, giving a total of 4 lattice points per unit cell \cite{wol85}.
The GYR phase is a continuous network periodic phase \cite{shi99}.
Those phases are shown in Figure \ref{fig4}.

\begin{figure}[!htbp]
\centering
\subfigure[A15]{
\includegraphics[width=2.5in]{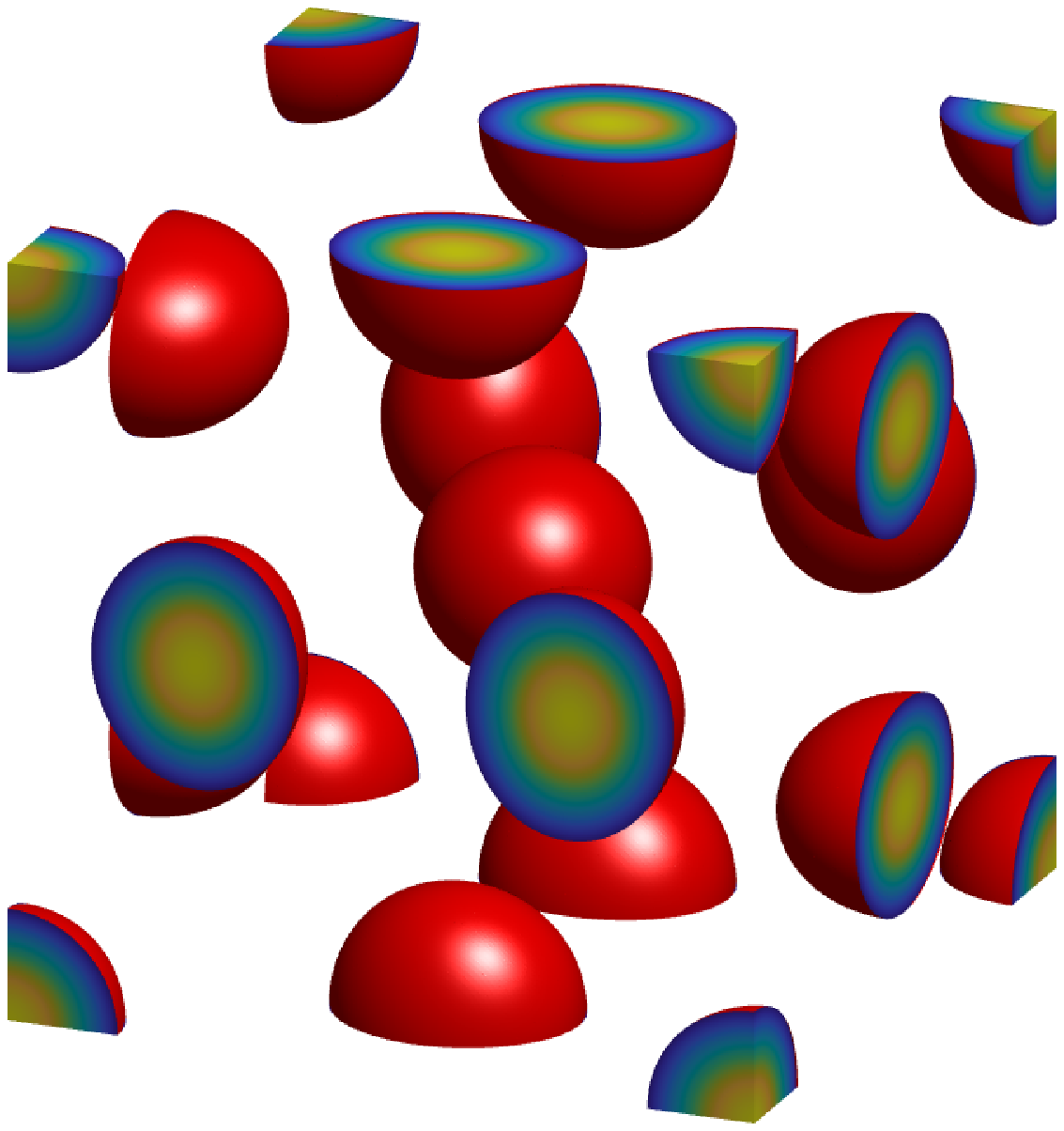}}
\subfigure[BCC]{
\includegraphics[width=2.5in]{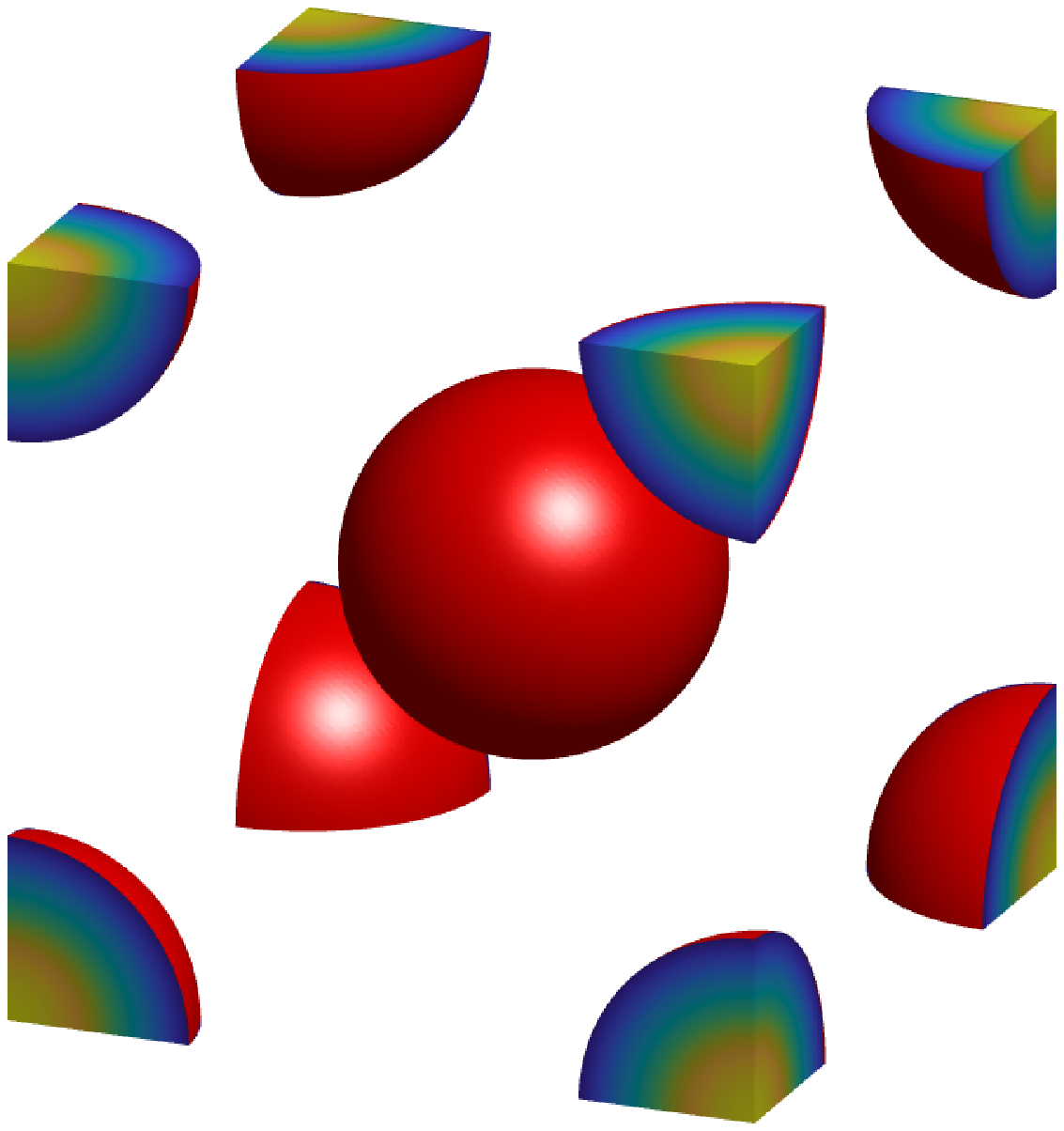}}
\subfigure[FCC]{
\includegraphics[width=2.5in]{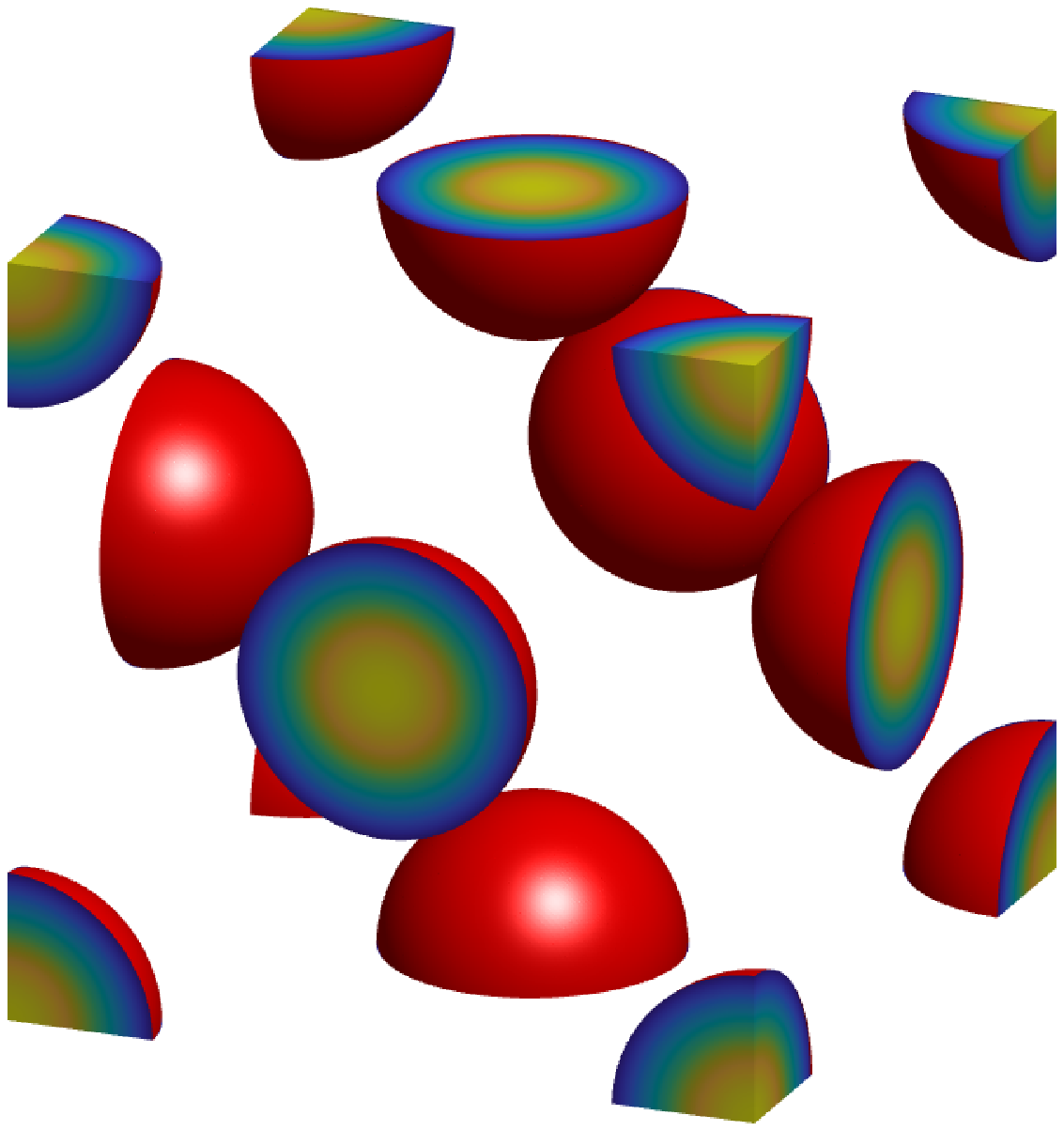}}
\subfigure[GYR]{
\includegraphics[width=2.5in]{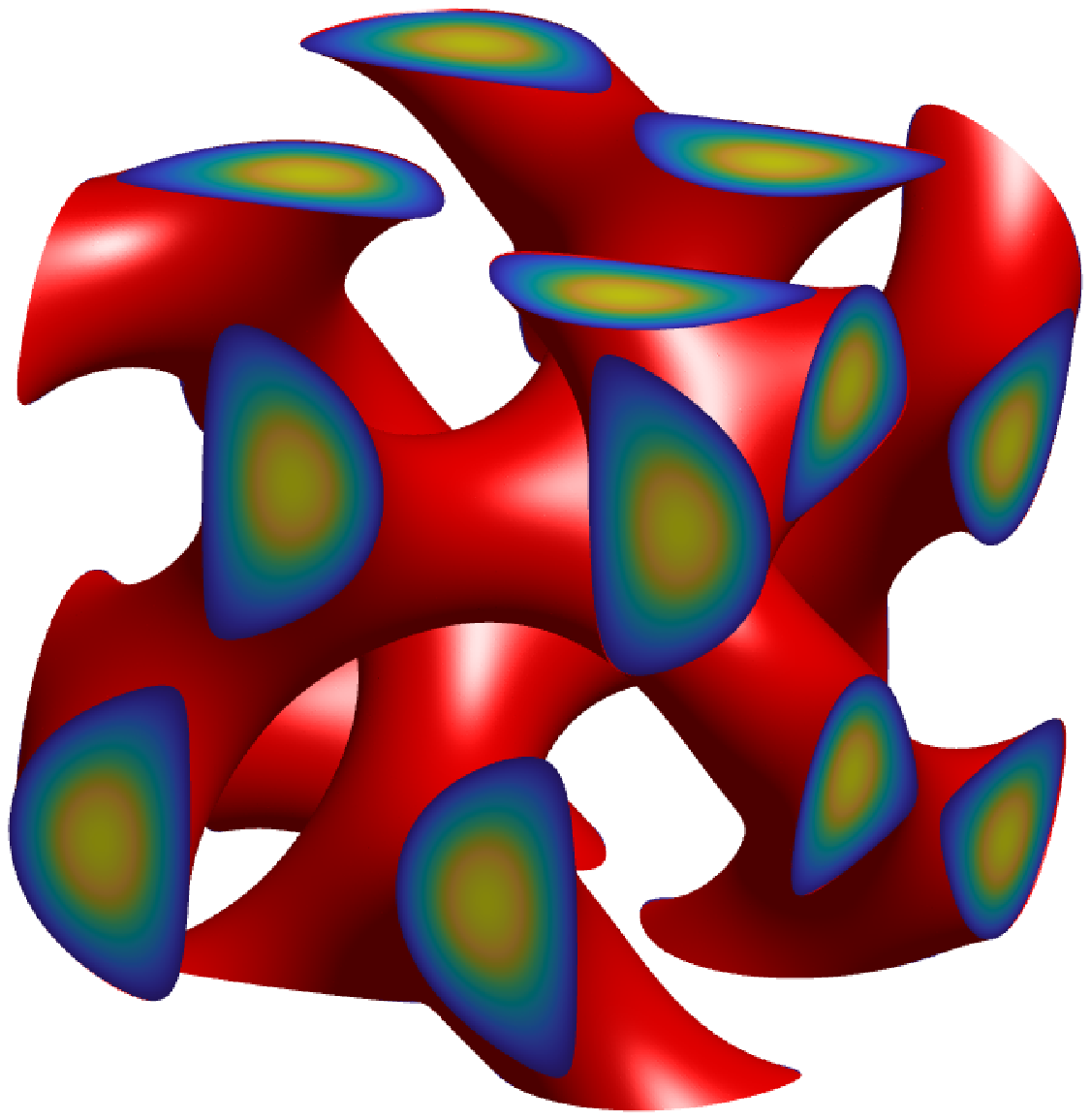}}
\caption{The three-dimensional periodic crystals in the LB model.}
\label{fig4}
\end{figure}

The computational domains in this example are defined by the unit cell $ \Omega=[0,a]^3 \subset\bbR^3$. 
Given a spatial grip $ \Cp $, the initial values are chosen as 
\begin{align*}
\phi^0(\Bx) =  \sum_{\Bk\in\Lambda^0} \wh{\phi}(\Bk) e^{i2\pi\Bk\cdot\Bx/a}\quad\forall\Bx\in\Cp, 
\end{align*}
where initial lattice points set $ \Lambda^0\subset\Ck $ only on which the Fourier coefficients located are nonzero.
The corresponding $ \Lambda^0 $ of those phases and the parameters in the LB model can be found in Table \ref{tab4}. 
For more details, please refer to \cite{jk13}.

\begin{table}[!htbp]
\centering
\caption{Initial lattice points set and parameters for three-dimensional periodic crystals.
$^o$ denotes the sign of Fourier coefficients is opposite.}
\label{tab4}
\vspace{10pt}
\begin{tabular}{|c|c|c|c|c|}
\hline
Phase                &                                   $ \Lambda^0 $ & $ a $                              &             $ \alpha $ & $ \gamma $            \\ \hline
\multirow{2}{*}{A15} &     $  (\pm2,\pm1, 0),(0,\pm2,1),(\pm1,0,2),  $ & \multirow{2}{*}{$2\sqrt{5}\pi $ }  &     \multirow{2}{*}{0} & \multirow{2}{*}{1.23} \\
                     & $   (\pm1,\pm2,0)^o,(\pm2,0,1)^o,(0,\pm1,2)^o $ &                                    &                        &                       \\ \hline
BCC                  &   $ (\pm1,\pm1,0),(\pm1,0,\pm1),(0,\pm1,\pm1) $ & $ 2\sqrt{2}\pi$                    &                      0 & 1.23                  \\ \hline
FCC                  &                               $ (\pm1,\pm1,1) $ & $ 2\sqrt{3}\pi$                    &                      0 & 2                     \\ \hline
\multirow{4}{*}{GYR} &                $ (1,-2,1),(1,2,-1),(-2,1,1),  $ & \multirow{4}{*}{$ 2\sqrt{6}\pi $ } & \multirow{4}{*}{0.47 } & \multirow{4}{*}{0.46} \\
                     &                $ (1,1, -2),(-1,1,2),(2,-1,1), $ &                                    &                        &                       \\
                     &             $ (1,2,1)^o,(-1,2,1)^o,(2,1,1)^o, $ &                                    &                        &                       \\
                     &             $ (1,1,2)^o,(1,-1,2)^o,(-2,1,1)^o $ &                                    &                        &                       \\ \hline
\end{tabular}
\end{table}

We use Legendre-Gauss-Lobatto quadrature points and change $ \Delta{t} $  to implement ASDC$_4^4$ algorithm.
The spatial grip $\Cp$ with $N=128$ is employed to compute the three-dimensional periodic crystals. 
The reference energies in Table \ref{tab5} are obtained via the spatial grip $\Cp$ with $N=256$.
Table \ref{tab5} shows the numerical results of ASDC$_4^4$ algorithm with the different step sizes.
More precisely, $ N_\text{iteration} $ and CPU time decrease with the increase of the time step $ \Delta{t} $. 
As Figure \ref{fig5} shows, ASDC$_M^K$ algorithm is also a monotone method for computing the three-dimensional periodic crystals.

\begin{table}[!htbp]
\centering
\caption{Numerical results of ASDC$_4^4$ algorithm for computing the three-dimensional periodic crystals.}
\label{tab5}
\vspace{10pt}
\begin{tabular}{|c|c|c|c|c|c|}
\hline

Phase                & $ \Delta{t} $ & $ N_\text{iteration}$ &    CPU Time (s) & $ \Ce_\Cp-\Ce_s $ &                            $ \Ce_s $ \\ \hline
\multirow{4}{*}{A15} &           0.1 & 643                   &          2970.6 & 9.2371e-13        &   \multirow{4}{*}{-57.4752889933902} \\ \cline{2-5}
                     &           0.5 & 141                   &          729.49 & 9.1660e-13        &                                      \\ \cline{2-5}
                     &             1 & 80                    &          471.26 & 9.1660e-13        &                                      \\ \cline{2-5}
                     &             2 & \textbf{51}           & \textbf{225.93} & 8.0291e-13        &                                      \\ \hline
\multirow{4}{*}{BCC} &           0.1 & 212                   &          1011.6 & 9.3614e-13        &   \multirow{4}{*}{-14.4932738221454} \\ \cline{2-5}
                     &           0.5 & 46                    &          317.59 & 9.8410e-13        &                                      \\ \cline{2-5}
                     &             1 & 26                    &          169.17 & 5.7376e-13        &                                      \\ \cline{2-5}
                     &             2 & \textbf{16}           &  \textbf{90.63} & 6.0574e-13        &                                      \\ \hline
\multirow{4}{*}{FCC} &           0.1 & 118                   &          518.26 & 7.3896e-13        &  \multirow{4}{*}{-209.6360921245683} \\ \cline{2-5}
                     &           0.5 & 26                    &          179.81 & 8.5265e-13        &                                      \\ \cline{2-5}
                     &             1 & 15                    &           67.62 & 2.8422e-13        &                                      \\ \cline{2-5}
                     &             2 & \textbf{9}            &  \textbf{61.79} & 2.2737e-13        &                                      \\ \hline
\multirow{5}{*}{GYR} &           0.1 & 435                   &          1998.8 & 9.3792e-13        & \multirow{5}{*}{-162.0665004168457 } \\ \cline{2-5}
                     &           0.5 & 94                    &           514.6 & 9.9476e-13        &                                      \\ \cline{2-5}
                     &             1 & 53                    &          329.04 & 6.5370e-13        &                                      \\ \cline{2-5}
                     &             2 & 33                    &          235.91 & 3.9790e-13        &                                      \\ \cline{2-5}
                     &             3 & \textbf{28}           &    \textbf{144} & 6.2528e-13        &                                      \\ \hline

\end{tabular}
\end{table}

\begin{figure}[!htbp]
\centering
\subfigure[A15]{
\includegraphics[width=2.5in]{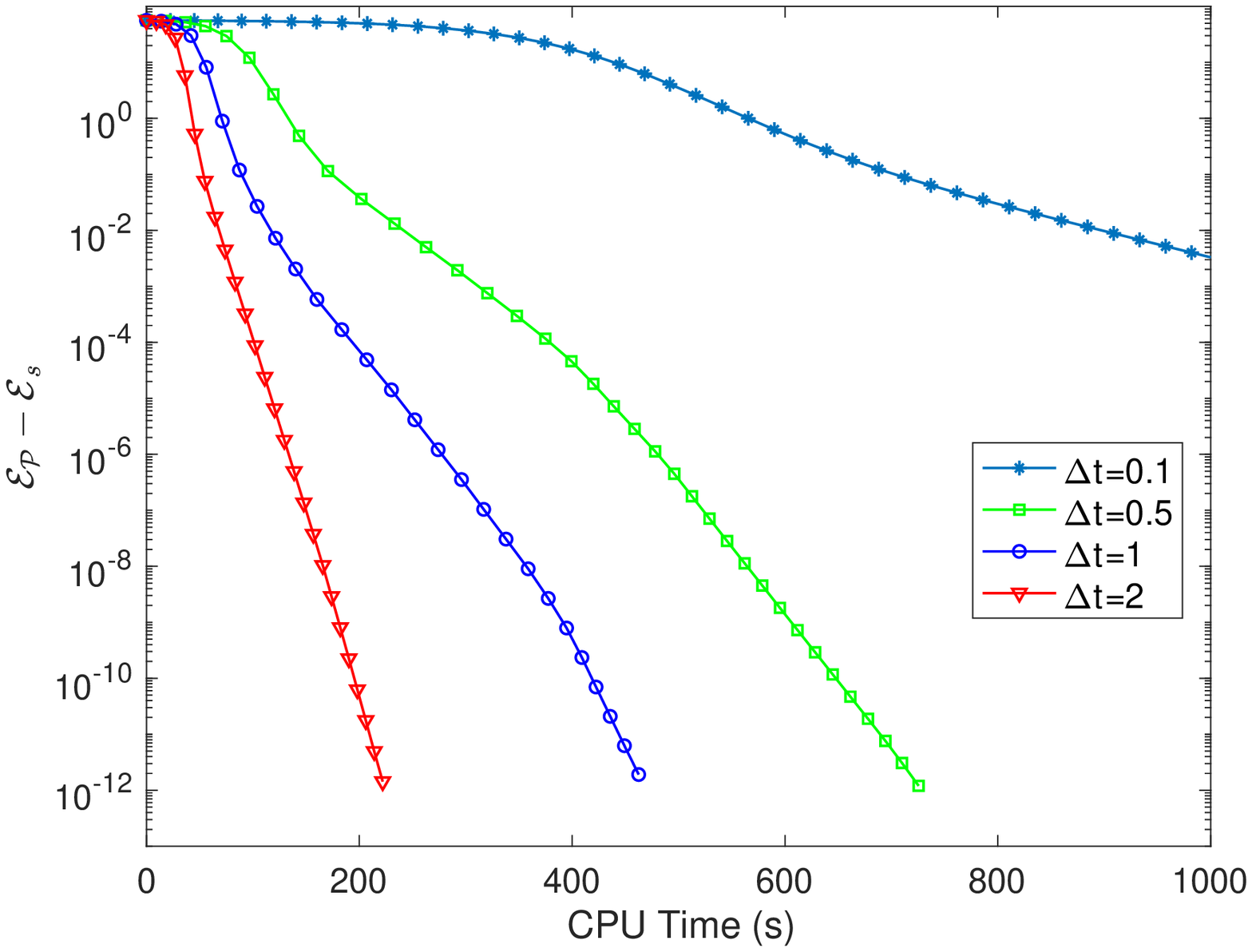}}
\subfigure[BCC]{
\includegraphics[width=2.5in]{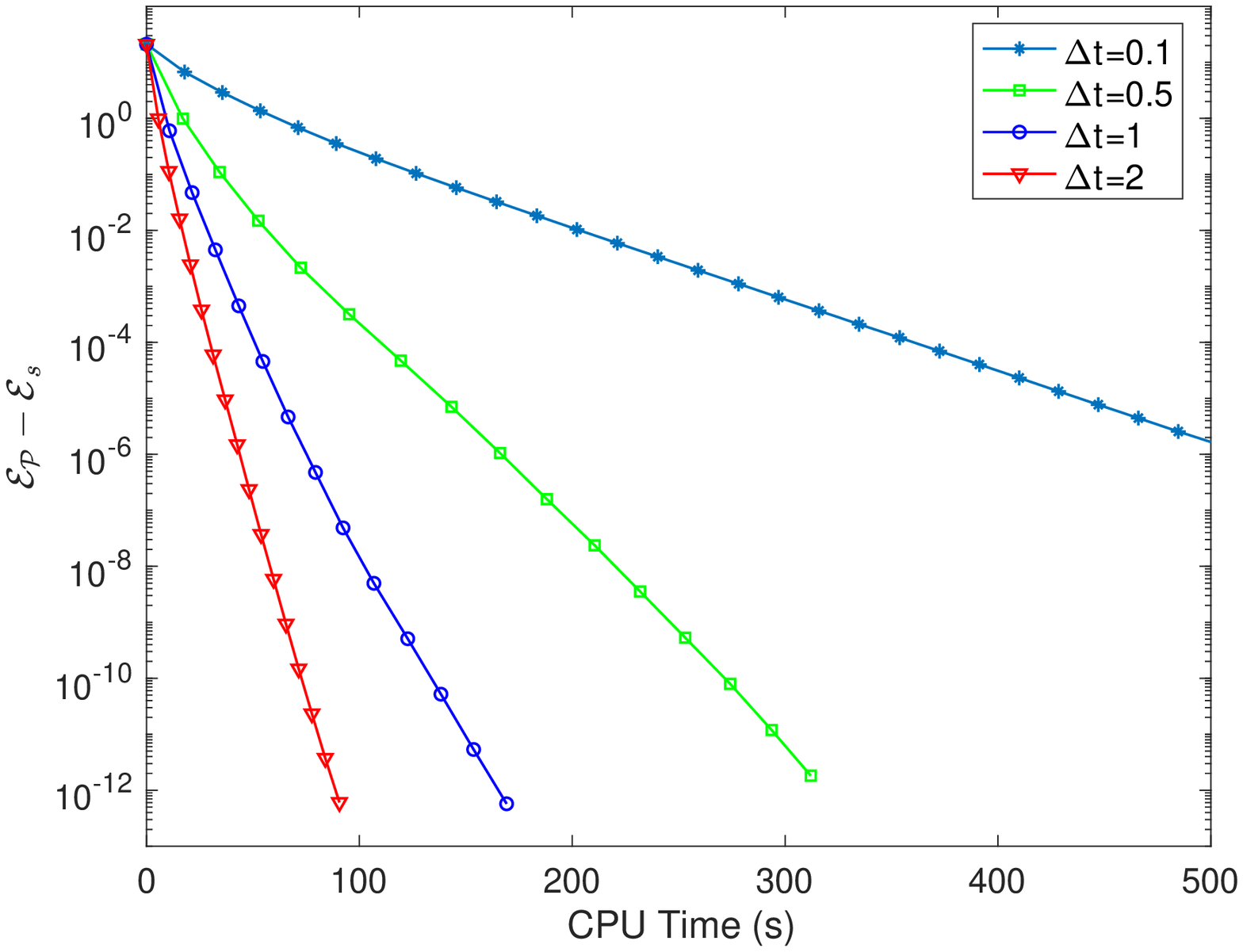}}
\subfigure[FCC]{
\includegraphics[width=2.5in]{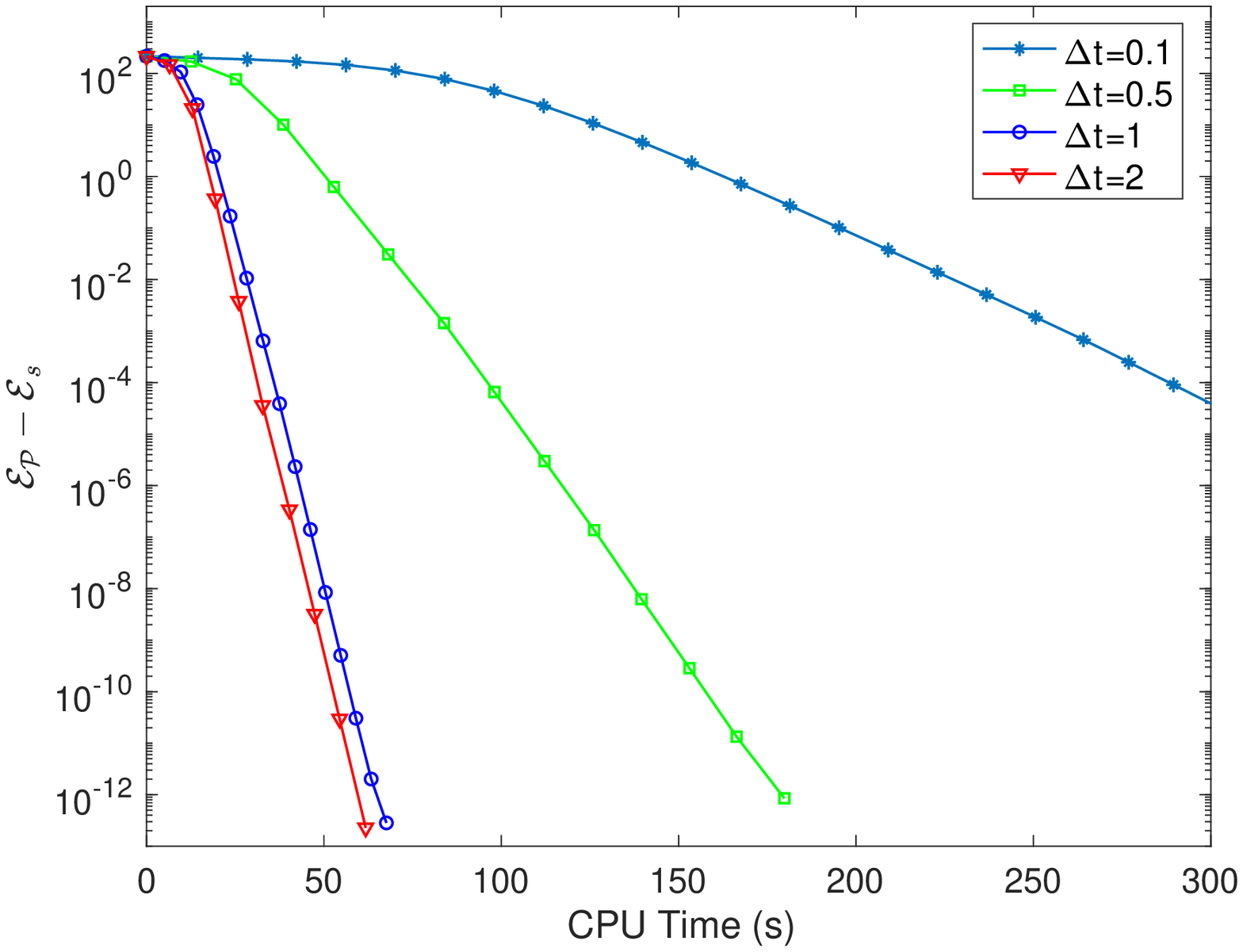}}
\subfigure[GYR]{
\includegraphics[width=2.5in]{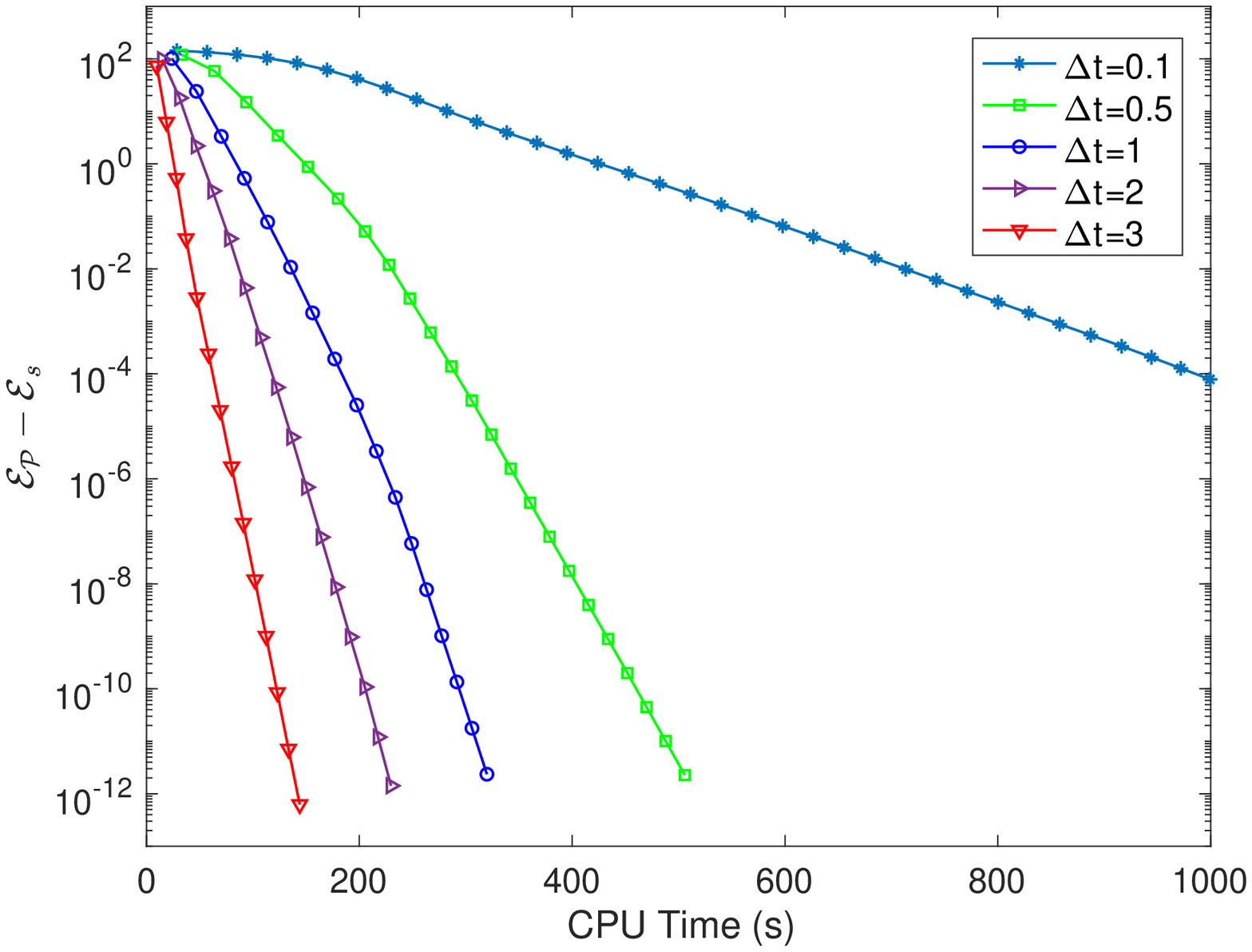}}
\caption{Energy difference over CPU Time of ASDC$_4^4$ algorithm for computing the three-dimensional periodic crystals.}
\label{fig5}
\end{figure}

\end{example}

\section{Conclusions}
This paper proposes an efficient numerical scheme to compute periodic crystals in the Landau--Brazovskii model by combining the SDC method with the linear convex splitting technique. 
Our algorithms can retain the energy dissipation and mass conservation properties during iteration.
An adaptive correction strategy is further implemented to reduce the cost time and improve the energy stability.
Numerical experiments for two and three dimensional periodic crystals are presented to show the efficiency and accuracy of the proposed method.

In the future, we will apply the SDC algorithm to the high-index saddle dynamics for efficient construction of solution landscape \cite{YinPRL,YinSCM,YinSISC}, which provides a pathway map including both stable minima and unstable saddle points. We may extend the developed numerical approach to the Lifshitz--Petrich (LP) model, which is widely used to compute quasiperiodic structures, such as the bi-frequency excited Faraday wave \cite{lif97}, and the phase transitions between crystals and quasicrystals \cite{yin21}. 

\section*{Acknowledgments}
This work was supported by the National Key Research and Development Program of China 2021YFF1200500 
and the National Natural Science Foundation of China 12225102, 12050002, and 12226316.


\end{document}